\pdfoutput=1
\RequirePackage{ifpdf}
\ifpdf 
\documentclass[pdftex]{sigma}
\else
\documentclass{sigma}
\fi

\numberwithin{equation}{section}

\newtheorem{Theorem}{Theorem}[section]
\newtheorem{Lemma}[Theorem]{Lemma}
\newtheorem{Proposition}[Theorem]{Proposition}
 { \theoremstyle{definition}
\newtheorem{Definition}[Theorem]{Definition}
\newtheorem{Remark}[Theorem]{Remark} }

\begin{document}


\newcommand{\arXivNumber}{1601.07194}

\renewcommand{\thefootnote}{$\star$}

\renewcommand{\PaperNumber}{090}

\FirstPageHeading

\ShortArticleName{Multivariate Orthogonal Polynomials and Modif\/ied Moment Functionals}

\ArticleName{Multivariate Orthogonal Polynomials \\ and Modif\/ied Moment Functionals\footnote{This paper is a~contribution to the Special Issue on Orthogonal Polynomials, Special Functions and Applications. The full collection is available at \href{http://www.emis.de/journals/SIGMA/OPSFA2015.html}{http://www.emis.de/journals/SIGMA/OPSFA2015.html}}}

\Author{Antonia M.~DELGADO, Lidia FERN\'ANDEZ, Teresa E.~P\'EREZ and Miguel A.~PI\~NAR}
\AuthorNameForHeading{A.M.~Delgado, L.~Fern\'andez, T.E.~P\'erez and M.A.~Pi\~nar}
\Address{IEMath -- Math Institute and Department of Applied Mathematics, University of Granada,\\ 18071, Granada, Spain}
\Email{\href{mailto:amdelgado@ugr.es}{amdelgado@ugr.es}, \href{mailto:lidiafr@ugr.es}{lidiafr@ugr.es}, \href{mailto:tperez@ugr.es}{tperez@ugr.es}, \href{mailto:mpinar@ugr.es}{mpinar@ugr.es}}
\URLaddress{\url{http://www.ugr.es/local/goya/}}

\ArticleDates{Received January 28, 2016, in f\/inal form September 05, 2016; Published online September 10, 2016}

\Abstract{Multivariate orthogonal polynomials can be introduced by using a moment functional def\/ined on the linear space of polynomials in several variables with real coef\/f\/icients. We study the so-called Uvarov and Christof\/fel modif\/ications obtained by adding to the moment functional a f\/inite set of mass points, or by multiplying it times a polynomial of total degree $2$, respectively. Orthogonal polynomials associated with modif\/ied moment func\-tio\-nals will be studied, as well as the impact of the modif\/ication in useful properties of the orthogonal polynomials. Finally, some illustrative examples will be given.}


\Classification{33C50; 42C10}

\renewcommand{\thefootnote}{\arabic{footnote}}
\setcounter{footnote}{0}

\section{Introduction}

Using a moment functional approach as in \cite{Ma85, Ma88, Ma91}, one interesting problem in the theory of orthogonal polynomials in one and several variables is the study of the modif\/ications of a quasi-def\/inite moment functional $u$ def\/ined on $\Pi$, the linear space of polynomials with real coef\/f\/icients. In fact, there are many works devoted to this topic since modif\/ications of moment functionals are underlying some well known facts, such as \textit{quasi-orthogonality}, relations between \textit{adjacent families}, quadrature and cubature formulas, higher-order (partial) dif\/ferential equations, etc.

In the univariate case, given a quasi-def\/inite moment functional $u$ def\/ined on $\Pi$, the (basic) {\it Uvarov modification} is def\/ined by means of the addition of a Dirac delta on $a\in\mathbb{R}$,
\begin{gather*}
v=u + \lambda \delta_a, \qquad \textrm{such that} \quad \langle v, p(x)\rangle =
\langle u, p(x)\rangle + \lambda p(a), \qquad \lambda\neq 0.
\end{gather*}

Apparently, this modif\/ication was introduced by V.B.~Uvarov in 1969~\cite{Uv69}, who studied the case where a f\/inite number of mass points is added to a measure, and proved connection formulas for orthogonal polynomials with respect to the modif\/ied measure in terms of those with respect to the original one.

In some special cases of classical Laguerre and Jacobi measures, if the perturbations are given at the end points of the support of the measure, then the new polynomials are eigenfunctions of higher-order dif\/ferential operators with polynomial coef\/f\/icients and they are called \textit{Krall polynomials} (see, for instance, \cite{Z99} and the references therein).

In the multivariate case, the addition of Dirac masses to a multivariate measure was studied in \cite{DFPPX10} and \cite{FPPX10}. Moreover, Uvarov modif\/ication for disk polynomials was analysed in \cite{DFPP12}.

Besides Uvarov modif\/ications by means of Dirac masses at f\/inite discrete set of points, in the context of several variables it is possible to modify the moment functional by means of moment functionals def\/ined on lower dimensional manifolds such as curves, surfaces, etc. Recently, a~family of orthogonal polynomials with respect to such a Uvarov modif\/ication of the classical ball measure by means of a mass uniformly distributed over the sphere was introduced in~\cite{MP16}. The authors proved that, at least in the Legendre case, these multivariate orthogonal polynomials satisfy a fourth-order partial dif\/ferential equation, which constitutes a natural extension of Krall orthogonal polynomials~\cite{Kr80} to the multivariate case. In~\cite{AM3}, a~modif\/ication of a~moment functional by adding another moment functional def\/ined on a~curve is presented, and a Christof\/fel formula built up in terms of a Fredholm integral equation is discussed. As far as we know, a general theory about Uvarov modif\/ications by means of moment functionals def\/ined on lower dimensional manifolds remains as an open problem.

Following \cite{Z97}, the univariate (basic) {\it Christoffel modification} is given by the multiplication of a moment functional $u$ times a~polynomial of degree~$1$, usually $x-a$, $a\in\mathbb{R}$,
\begin{gather*}
v=(x-a) u, \quad \textrm{acting as} \quad \langle v, p(x)\rangle = \langle u, (x-a) p(x)\rangle.
\end{gather*}
This type of transformations were f\/irst considered by E.B.~Christof\/fel in 1858 \cite{Ch1858} within the framework of Gaussian quadrature theory. Nowadays, Christof\/fel formulas are classical results in the theory of orthogonal polynomials, and they are presented in many general references (e.g.,~\cite{Ch78, Sz78}).

Christof\/fel modif\/ication is characterized by the {\it linear relations} that both families, modif\/ied and non-modif\/ied orthogonal polynomials, satisfy. They are called \textit{connection formulas}. It is well known that some families of classical orthogonal polynomials can be expressed as linear combinations of polynomials of the same family for dif\/ferent values of their parameters, the so-called \emph{relations between adjacent families} (e.g., see formulas in Chapter~22 in~\cite{AS72} for Jacobi polynomials, or~(5.1.13) in~\cite{Sz78} for Laguerre polynomials). The study of such type of linear combinations is also related with the concept of \emph{quasi-orthogonality} introduced by M.~Riesz in 1921 (see \cite[p.~64]{Ch78}) as the basis of his analysis of the moment problem, and it is related with quadrature formulas based on the zeros of orthogonal polynomials.

The extension of this kind of results to the multivariate case is not always possible. Gaussian cubature formulas of degree $2n-1$ were characterized by Mysovskikh~\cite{M81} in terms of the number of common zeros of the multivariate orthogonal polynomials. However, these formulas only exist in
very special cases and the case of degree $2n-2$ becomes interesting. Here, linear combinations of multivariate orthogonal polynomials play an important role, as it can be seen for instance in \cite{BSX95,SX94,X92,X94}.

To our best knowledge, one of the f\/irst studies about Christof\/fel modif\/ications in several variables, multiplying a moment functional times a polynomial of degree $1$, was done in \cite{APPR14}, where the modif\/ication naturally appears in the study of {\it linear relations} between two families of multivariate polynomials. Necessary and suf\/f\/icient conditions about the existence of orthogonality properties for one of the families was given in terms of the three term relations, by using Favard's theorem in several variables~\cite{DX14}.

In \cite{BM04} the authors show that modif\/ications of univariate moment functionals are related with the Darboux factorization of the associated Jacobi matrix. In this direction, in \cite{AM16, AM2, AM3} long discussions about several aspects of the theory of multivariate orthogonal polynomials can be found. In particular, Darboux transformations for orthogonal polynomials in several variables were presented in~\cite{AM16}, and in~\cite{AM2} we can f\/ind an extension of the univariate Christof\/fel determinantal formula to the multivariate context. Also, as in the univariate case, they proved a~connection with the Darboux factorization of the Jacobi block matrix associated with the three term recurrence relations for multivariate orthogonal polynomials. Similar considerations for multivariate Geronimus and more general linear spectral transformations of moment functionals can be found, among other topics, in~\cite{AM3}.

In this paper, we study Uvarov and Christof\/fel modif\/ications of quasi-def\/inite moment functionals. The study of orthogonal polynomials associated with moment functionals f\/its into a quite general frame that includes families of orthogonal polynomials associated either with positive-def\/inite or non positive-def\/inite moment functionals such as those generated using Bessel polynomials, among others. We give necessary and suf\/f\/icient conditions in order to obtain the quasi-def\/initeness of the modif\/ied moment functional in both cases, Uvarov and Christof\/fel modif\/ications (see Theorems~\ref{main-theorem} and~\ref{T42}). We also investigate properties of the polynomials associated with the modif\/ied functional, relations between original and modif\/ied orthogonal polynomials as well as the impact of the modif\/ication in some useful properties of the orthogonal polynomials.

When dealing with the Christof\/fel modif\/ication, and in the case where both moment functionals, the original and the modif\/ied, are quasi-def\/inite, some of the results are similar to those obtained in~\cite{AM16, AM2, AM3} using a dif\/ferent technique. In particular, the necessary condition in our Theorem~\ref{T42} was proven there for an arbitrary degree polynomial, but the suf\/f\/icient condition was not discussed there.

The main results of this work can be divided in three parts, corresponding with Sections~\ref{section3},~\ref{section4} and~\ref{section5}. Section~\ref{section2} is devoted to establish the basic concepts and tools we will need along the paper. For that section, we recall the standard notations and basic results in the theory of multivariate orthogonal polynomials following mainly \cite{DX14}.

Uvarov modif\/ication of a quasi-def\/inite moment functional is studied in Section~\ref{section3}. In this case, we modify a quasi-def\/inite moment functional by adding several Dirac deltas at a f\/inite discrete set of f\/ixed points. First, we will give a necessary and suf\/f\/icient condition for the quasi-def\/initeness of the moment functional associated with this Uvarov modif\/ication, and, in the af\/f\/irmative case, we will deduce the connection between both families of orthogonal polynomials. A similar study was done in~\cite{DFPPX10} in the case when the original moment functional is def\/ined from a measure, and the modif\/ication is def\/ined by means of a positive semi-def\/inite matrix. In that case, both moment functionals are positive-def\/inite, and both orthogonal polynomial systems exist. As a consequence, since it is possible to work with orthonormal polynomials, in~\cite{DFPPX10} the f\/irst family of orthogonal polynomials is considered orthonormal, and formulas are simpler than in the general case considered here.

\looseness=1 In Section~\ref{section4} we study the Christof\/fel modif\/ication by means of a second degree polynomial. This is not a trivial extension of the case when the degree of the polynomial is~$1$ studied in~\cite{APPR14}, since in several variables not every polynomial of degree $2$ factorizes as a~product of polynomials of degree~$1$. Again, we relate both families of orthogonal polynomials and also we deduce the orthogonality by using Favard's theorem in several variables. In fact, from a~recursive expression for the modif\/ied polynomials, we give necessary and suf\/f\/icient conditions for the existence of a~three term relation, and we show that there exists a polynomial of second degree constructed in terms of the f\/irst connection coef\/f\/icients. Since three term relations can be reformulated in terms of Jacobi block matrices, similar connection results can be found in~\cite{AM2,AM3} for arbitrary degree polynomials by using a block matrix formalism. Moreover, we study this kind of Christof\/fel modif\/ication in the particular case when the original moment functional is centrally symmetric, the natural extension of the concept of symmetry for univariate moment functionals.

Finally, Section~\ref{section5} is devoted to apply our results to two families of multivariate orthogonal polynomials: the classical orthogonal polynomials on the ball in $\mathbb{R}^d$, and the Bessel--Laguerre polynomials, a family of bivariate polynomials orthogonal with respect to a non positive-def\/inite moment functional, that can be found in~\cite{KLL01} as solution of one of the classical Krall and Shef\/fer's second-order partial dif\/ferential equation~\cite{KS67}.

First, we modify the classical moment functional on the ball by adding a Dirac mass at $\mathbf{0}\in\mathbb{R}^d$. This example was introduced in~\cite{FPPX10}, and here, we complete that study giving, as a new result, the asymptotics of the modif\/ication. Next, we use the Christof\/fel
modif\/ication by means of a second degree polynomial to deduce a {\it relation between adjacent families} of classical ball orthogonal polynomials in several variables. As far as we know, there is not a relation of this kind for ball polynomials in the literature. This relation can be seen as an extension of~(22.7.23) in~\cite[p.~782]{AS72} for Gegenbauer polynomials in one variable. Last example corresponds to a Uvarov modif\/ication for the non positive-def\/inite classical Bessel--Laguerre bivariate polynomials def\/ined in~\cite{KLL01}.

\section{Basic tools}\label{section2}

In this section we collect the def\/initions and basic tools about orthogonal polynomials in several variables that will be used later. We follow mainly~\cite{DX14}.

Along this paper, we will denote by $\mathcal{M}_{h \times k}(\mathbb{R})$ the linear space of matrices of size $h\times k$ with real entries, and the notation will simplify when $h=k$ as $\mathcal{M}_{h}(\mathbb{R})$. As usual, we will say that $M\in\mathcal{M}_h(\mathbb{R})$ is non-singular
(or invertible) if $\det M\neq 0$, and symmetric if $M^t = M$, where $M^t$ denotes its transpose. Moreover, $I_h$ will denote the identity matrix of size $h$, and we will omit the subscript when the size is clear from the context.

Let us consider the $d$-dimensional space $\mathbb{R}^d$, with $d \geqslant 1$. Given $\nu = (n_1, n_2, \ldots, n_d) \in \mathbb{N}_0^d$ a~multi-index, a monomial in $d$ variables is an expression of the form
\begin{gather*}
\mathbf{x}^\nu = x_1^{n_1} x_2^{n_2} \cdots x_d^{n_d},
\end{gather*}
where $\mathbf{x} = (x_1, \dots, x_d) \in \mathbb{R}^d$. The total degree of the monomial is denoted by $n = |\nu| = n_1 + n_2 + \cdots + n_d$. Then, a~polynomial of total degree $n$ in $d$ variables with real coef\/f\/icients is a f\/inite linear combination of monomials of degree at most $n$,
\begin{gather*}
p(\mathbf{x}) = \sum_{|\nu|\leqslant n} a_\nu \mathbf{x}^\nu, \qquad a_\nu \in \mathbb{R}.
\end{gather*}
Let us denote by $\Pi^d_n$ the linear space of polynomials in $d$ variables of total degree less than or equal to~$n$, and by $\Pi^d$ the linear space of all polynomials in $d$ variables.

For $d\geqslant2$ and $n\geqslant0$, if we denote
\begin{gather*}
r^d_n = \# \{ \mathbf{x}^\nu \colon |\nu| = n \} = \binom{n+d-1}{n},
\end{gather*}
then, unlike the univariate case, $r_n^d > 1$ for $d\geqslant2$. Moreover,
\begin{gather}\label{rbold}
\mathbf{r}_n^d = \dim \Pi^d_n = \sum_{m=0}^n r^d_m = \binom{n+d}{n}.
\end{gather}
When we deal with more than one variable, the f\/irst problem we have to face is that there is not a natural way to order the monomials. As in \cite{DX14}, we use the {\it graded lexicographical order}, that is, we order the monomial by their total degree, and then by reverse lexicographical order. For instance, if $d=2$, the order of the monomials is
\begin{gather*}
\big\{ 1; x_1, x_2; x_1^2, x_1 x_2, x_2^2; x_1^3, x_1^2 x_2, x_1 x_2^2, x_2^3; \dots \big\}.
\end{gather*}

A useful tool in the theory of orthogonal polynomials in several variables is the representation of a~basis of polynomials as a {\it polynomial system}~(PS).

\begin{Definition}
A \emph{polynomial system $($PS$)$} is a sequence of column vectors of increasing size~$r_n^d$, $\{\mathbb{P}_n\}_{n \geqslant 0}$, whose entries are independent polynomials of total degree~$n$
\begin{gather*}
\mathbb{P}_n = \mathbb{P}_n(\mathbf{x}) = \big( P_\nu^n(\mathbf{x}) \big)_{|\nu|=n} = \big(
P_{\nu_1}^n(\mathbf{x}), P_{\nu_2}^n(\mathbf{x}), \dots , P_{\nu_{r^d_n}}^n(\mathbf{x})\big)^t,
\end{gather*}
where $\nu_1, \nu_2, \dots, \nu_{r^d_n} \in \big\{ \nu \in \mathbb{N}_0^d \colon |\nu|=n \big\}$ are dif\/ferent multi-indexes arranged in the reverse lexicographical order.
\end{Definition}

Observe that, for $n\geqslant 0$, the entries of $\{\mathbb{P}_0, \mathbb{P}_1, \dots, \mathbb{P}_n \}$ form a~basis of~$\Pi_n^d$, and, by extension, we will say that the vector polynomials $\{\mathbb{P}_m\}_{m=0}^n$ is a~basis of~$\Pi_n^d$.

Using this representation, we can def\/ine the \emph{canonical polynomial system}, as the sequence of vector polynomials whose entries are the basic monomials arranged in the reverse lexicographical order,
\begin{gather*}
\{ \mathbb{X}_n \}_{n\geqslant0} = \big\{ \big(\mathbf{x}^{\nu_1}, \mathbf{x}^{\nu_2}, \dots, \mathbf{x}^{\nu_{r_n^d}}\big)^t \colon |\nu_i| =n \big\}_{n\geqslant0}.
\end{gather*}
Thus, each polynomial vector $\mathbb{P}_n$, for $n\geqslant 0$, can be expressed as a unique linear combination of the canonical polynomial system with matrix coef\/f\/icients
\begin{gather*}
\mathbb{P}_n = G_{n,n} \mathbb{X}_n + G_{n,n-1} \mathbb{X}_{n-1} + \cdots + G_{n,0} \mathbb{X}_0,
\end{gather*}
where $G_{n,k} \in \mathcal{M}_{r_n^d \times r_k^d} (\mathbb{R})$ for $k=0, 1, \ldots, n$. Since both vectors $\mathbb{P}_n$ and $\mathbb{X}_n$ contain a system of independent polynomials, the square matrix $G_{n,n}$ is non-singular and it is called the {\it leading coefficient} of the vector polynomial $\mathbb{P}_n$. Usually, it will be denoted by
\begin{gather*}
G_n(\mathbb{P}_n) = G_{n,n} \in \mathcal{M}_{r_n^d}(\mathbb{R}).
\end{gather*}
In the case when the leading coef\/f\/icient is the identity matrix for all vector polynomial in a PS
\begin{gather*}
G_n(\mathbb{P}_n) = I_{r_n^d}, \qquad n\geqslant 0,
\end{gather*}
then we will say that $\{\mathbb{P}_n\}_{n \geqslant 0}$ is a {\it monic PS}. With this notation, we will say that two vector polynomials $\mathbb{P}_n$ and $\mathbb{Q}_n$, for $n\geqslant 0$, have the same leading coef\/f\/icient if $G_n(\mathbb{P}_n) = G_n(\mathbb{Q}_n)$, or, equivalently, if any entry of the vector $\mathbb{P}_n - \mathbb{Q}_n$ is a polynomial of degree at most $n-1$.

The \emph{shift operator} in several variables, that is, the multiplication of a polynomial times a~va\-riable $x_i$ for $i=1, 2, \ldots, d$, can be expressed in terms of the canonical PS as
\begin{gather*}
x_i \mathbb{X}_n = L_{n,i} \mathbb{X}_{n+1}, \qquad n\geqslant 0,
\end{gather*}
where $L_{n,i}$, $i=1, 2, \ldots, d$ are $r_n^d\times r^d_{n+1}$ full rank matrices such that (see~\cite{DX14})
\begin{gather*}
L_{n,i} L^t_{n,i} = I_{r^d_n}.
\end{gather*}
In particular, $L_{n,i}$ is a matrix containing the columns of the identity matrix of size~$r_n^d$ but including $r^d_{n+1}-r_n^d$ columns of zeros eventually separating the columns of the identity.

Moreover, for $i, j = 1, 2, \ldots, d$, the matrix products $L_{n,i} L_{n+1,j}\in \mathcal{M}_{r_n^d\times r^d_{n+2}}(\mathbb{R})$, are also full rank matrices and of the same type of $L_{n,i}$. In addition, since $x_i x_j \mathbb{X}_n = x_j x_i \mathbb{X}_n$, for $n\geqslant0$, we get
\begin{gather}\label{L-i-j}
L_{n,i} L_{n+1,j} = L_{n,j} L_{n+1,i}, \qquad i, j = 1, 2, \ldots, d, \qquad n\geqslant0.
\end{gather}
For instance, if $d=2$,
\begin{gather*}
L_{n,1} = \left(
\begin{array}{@{}ccc|c@{}}
1 & & \bigcirc & 0 \\
& \ddots & & \vdots \\
\bigcirc & & 1 & 0 \\
\end{array}
\right), \qquad
L_{n,2} = \left(
\begin{array}{@{}c|ccc@{}}
0 & 1 & & \bigcirc \\
\vdots & & \ddots & \\
0 & \bigcirc & & 1 \\
\end{array}
\right),
\end{gather*}
and
\begin{gather*}
L_{n,1} L_{n+1,1} = \left(
\begin{array}{@{}ccc|cc@{}}
1 & & \bigcirc & 0 & 0 \\
 & \ddots & & \vdots & \vdots \\
\bigcirc & & 1 & 0 & 0 \\
\end{array}
\right),\qquad
L_{n,2} L_{n+1,2} = \left(
\begin{array}{@{}cc|ccc@{}}
0 & 0 & 1 & & \bigcirc \\
\vdots & \vdots & & \ddots & \\
0 & 0 & \bigcirc & & 1 \\
\end{array}
\right),
\\
L_{n,1} L_{n+1,2} = \left(
\begin{array}{@{}c|ccc|c@{}}
0 & 1 & & \bigcirc & 0 \\
\vdots & & \ddots & & \vdots \\
0 & \bigcirc & & 1 & 0 \\
\end{array}
\right) = L_{n,2} L_{n+1,1}.
\end{gather*}

Let us now turn to deal with moment functionals and orthogonality in several variables. Given a~sequence of real numbers $\{\mu_{\nu}\}_{|\nu|=n\geqslant0}$, the moment functional~$u$ is def\/ined by means of its moments
\begin{align*}
u \colon \ \Pi^d & \longrightarrow \mathbb{R},\\
 \mathbf{x}^\nu & \longmapsto \langle u, \mathbf{x}^\nu \rangle = \mu_\nu,
\end{align*}
and extended to polynomials by linearity, i.e., if $p(\mathbf{x}) = \sum\limits_{|\nu|\leqslant n} a_\nu \mathbf{x}^\nu$, then
\begin{gather*}
\langle u, p(\mathbf{x})\rangle = \sum_{|\nu|\leqslant n} a_\nu \mu_\nu.
\end{gather*}
Now, we recall some basic operations acting over a moment functional $u$. The {\it action of~$u$ over a polynomial matrix} is def\/ined by
\begin{gather*}
\langle u, M \rangle =\left(\langle u, m_{i,j}(\mathbf{x})\rangle \right)_{i,j=1}^{h,k}\in {\mathcal M}_{h\times k}(\mathbb{R}),
\end{gather*}
where $M=(m_{i,j}(\mathbf{x}))_{i,j=1}^{h,k} \in {\mathcal M}_{h\times k}(\Pi^d)$, and the {\it left product of a polynomial $p\in \Pi^d$ times $u$} by
\begin{gather*}
\langle p u, q \rangle = \langle u, p q \rangle,\qquad \forall\, q\in \Pi^d.
\end{gather*}

Using the canonical polynomial system, for $h,k \geqslant 0$, we def\/ine the $r_h^d\times r_k^d$ {\it block of moments}
\begin{gather*}
\textbf{m}_{h,k} = \langle u, \mathbb{X}_h \mathbb{X}^t_k\rangle = \textbf{m}_{k,h}^t.
\end{gather*}
We must remark that $\textbf{m}_{h,k}$ contains all of moments of order $h\times k$. Then, we can see the moment matrix as a block matrix in the form
\begin{gather*}
\textbf{M}_n = \left(\begin{array}{@{}c|c|c|c|c@{}}
\textbf{m}_{0,0} & \textbf{m}_{0,1} & \textbf{m}_{0,2} & \cdots & \textbf{m}_{0,n}\\
\hline
\textbf{m}_{1,0} & \textbf{m}_{1,1} & \textbf{m}_{1,2} & \cdots & \textbf{m}_{1,n}\\
\hline
\textbf{m}_{2,0} & \textbf{m}_{2,1} & \textbf{m}_{2,2} & \cdots & \textbf{m}_{2,n}\\
\hline
\vdots & \vdots & \vdots & & \vdots \\
\hline
\textbf{m}_{n,0} & \textbf{m}_{n,1} & \textbf{m}_{n,2} & \cdots & \textbf{m}_{n,n}
\end{array}\right)
\end{gather*}
of dimension $\mathbf{r}_n^d$ def\/ined in \eqref{rbold}.

\begin{Definition}
A moment functional $u$ is called quasi-def\/inite or regular if and only if
\begin{gather*}
\Delta_n^d = \det \textbf{M}_n \neq 0, \qquad n\geqslant 0.
\end{gather*}
\end{Definition}

Now, we are ready to introduce the orthogonality. Two polynomials $p, q \in \Pi^d$ are said to be orthogonal with respect to $u$ if $\langle u, p q \rangle = 0$. A given polynomial $p \in \Pi^d_n$ of exact degree $n$ is an orthogonal polynomial if it is orthogonal to any polynomial of lower degree.

We can also introduce the orthogonality in terms of a polynomial system. We will say that a~PS $\{\mathbb{P}_n\}_{n\geqslant 0}$ is orthogonal (OPS) with respect to~$u$ if
\begin{gather*}
\langle u, \mathbb{X}_{n} \mathbb{P}_{n}^t \rangle = S_{n},\quad n\geqslant 0, \qquad \langle u, \mathbb{X}_{m} \mathbb{P}_{n}^t \rangle = 0, \quad m < n,
\end{gather*}
where $S_{n}$ is a non-singular matrix of size $r_n^d \times r_n^d$. As a consequence, it is clear that
\begin{gather*}
\langle u, \mathbb{P}_{n} \mathbb{P}_{n}^t \rangle = H_{n}, \quad n\geqslant 0, \qquad \langle u, \mathbb{P}_{m} \mathbb{P}_{n}^t \rangle = 0, \quad m \neq
n,
\end{gather*}
where the matrix $H_{n}$ is symmetric and non-singular. At this point we have to notice that, with this def\/inition, the orthogonality between the polynomials of the same total degree may not hold. In the case when the matrix~$H_n$ is diagonal for all $n\geqslant0$ we say that the OPS is mutually orthogonal.

A moment functional $u$ is \emph{quasi-definite or regular} if and only if there exists an OPS. If $u$ is quasi-def\/inite then there exists a unique \emph{monic OPS}. As usual $u$ is said \emph{positive definite} if $\langle u, p^2\rangle >0$, for all polynomial $p\neq 0$, and a positive
def\/inite moment functional $u$ is quasi-def\/inite. In this case, there exist an \emph{orthonormal basis}
satisfying
\begin{gather*}
\langle u, \mathbb{P}_{n} \mathbb{P}_{n}^t \rangle = H_n = I_{r_n^d}.
\end{gather*}

For $n\geqslant0$, the {\it kernel functions} in several variables are the symmetric functions def\/ined by
\begin{gather}
P_m(u; \mathbf{x}, \mathbf{y}) = \mathbb{P}_m(\mathbf{x})^t H_m^{-1} \mathbb{P}_m(\mathbf{y}), \qquad m\geqslant0,\nonumber\\
K_n(u; \mathbf{x}, \mathbf{y}) = \sum_{m=0}^n \mathbb{P}_m(\mathbf{x})^t H_m^{-1} \mathbb{P}_m(\mathbf{y})
= \sum_{m=0}^n P_m(u; \mathbf{x}, \mathbf{y}). \label{kernelK}
\end{gather}
Both kernels satisfy the usual reproducing property, and they are independent of the particularly chosen orthogonal polynomial system. For $n=0$, we assume $P_{-1}(u;\mathbf{x},\mathbf{y})=K_{-1}(u;\mathbf{x}, \mathbf{y}) = 0$.

To f\/inish this section, we need to recall the \emph{three term relations} satisf\/ied by orthogonal polynomials in several variables. As in the univariate case, the orthogonality can be characterized by means of the \emph{three term relations} \cite[p.~74]{DX14}.

\begin{Theorem}
Let $\{\mathbb{P}_n\}_{n\geqslant 0} =\{P^n_\nu(\mathbf{x})\colon |\nu| = n, n\geqslant0\}$, $\mathbb{P}_0=1$, be an arbitrary sequence in~$\Pi^d$. Then the following statements are equivalent.
\begin{enumerate}\itemsep=0pt
\item[$(1)$] There exists a linear functional $u$ which defines a quasi-definite moment functional on~$\Pi^d$ and which makes $\{\mathbb{P}_n\}_{n\geqslant0}$ an orthogonal basis in~$\Pi^d$.

\item[$(2)$] For $n\geqslant 0$, $1\leqslant i\leqslant d$, there exist matrices $A_{n,i}$, $B_{n,i}$ and $C_{n,i}$ of respective sizes
$r_n^d\times r_{n+1}^d$, $r_n^d\times r_{n}^d$ and $ r_n^d \times r_{n-1}^d$, such that
\begin{enumerate}\itemsep=0pt
\item[$(a)$] the polynomials $\mathbb{P}_n$ satisfy the three term relations
\begin{gather} \label{3TR}
 x_i \mathbb{P}_n = A_{n,i} \mathbb{P}_{n+1} + B_{n,i} \mathbb{P}_{n} +C_{n,i} \mathbb{P}_{n-1}, \qquad 1\leqslant i\leqslant d,
\end{gather}
with $\mathbb{P}_{-1}=0$ and $C_{-1,i} =0$,

\item[$(b)$] for $n\geqslant 0$ and $1\leqslant i\leqslant d$, the matrices $A_{n,i}$ and $C_{n+1,i}$
satisfy the rank conditions
\begin{gather}\label{cond_rank_1}
\operatorname{rank} A_{n,i} = \operatorname{rank} C_{n+1,i}=r_n^d,
\end{gather}
and,
\begin{gather}\label{cond_rank_2}
\operatorname{rank} A_{n} = \operatorname{rank} C_{n+1}^t = r_{n+1}^d,
\end{gather}
where $A_n$ is the joint matrix of $A_{n,i}$, defined as
\begin{gather*}
A_n = \big(A_{n,1}^t,A_{n,2}^t, \ldots, A_{n,d}^t\big)^t\in \mathcal{M}_{dr_n^d\times r_{n+1}^d}(\mathbb{R}),
\end{gather*}
and $C_{n+1}^t$ is the joint matrix of $C_{n+1,i}^t$. Moreover,
\begin{gather}
 A_{n,i} H_{n+1} = \langle u, x_i \mathbb{P}_n \mathbb{P}_{n+1}^t\rangle,\qquad
 B_{n,i} H_n = \langle u, x_i \mathbb{P}_n \mathbb{P}_{n}^t\rangle,\nonumber\\
 C_{n,i} H_{n-1} = \langle u, x_i \mathbb{P}_n \mathbb{P}_{n-1}^t\rangle =H_n A^t_{n-1,i}.\label{C1}
\end{gather}
\end{enumerate}
\end{enumerate}
\end{Theorem}

In the case when the orthogonal polynomial system is monic, it follows that $A_{n,i} = L_{n,i}$, $n\geqslant0$ for $1\leqslant i\leqslant d$. In this case, the rank conditions for the matrices $A_{n,i} = L_{n,i}$ and $A_{n} = L_{n}$ obviously hold.

\section{Uvarov modif\/ication}\label{section3}

Let $u$ be a quasi-def\/inite moment functional def\/ined on $\Pi^d$, let $\{\xi_1, \xi_2, \ldots, \xi_N\}$ be a f\/ixed set of distinct points in $\mathbb{R}^d$, and let $\{\lambda_1, \lambda_2, \ldots, \lambda_N\}$ be a f\/inite set of non zero real numbers. Then, for $p \in \Pi^d$, the expression
\begin{gather} \label{v-def}
\langle v, p \rangle = \langle u, p \rangle + \lambda_1 p(\xi_1) + \lambda_2 p(\xi_2) + \cdots + \lambda_N p(\xi_N),
\end{gather}
def\/ines a moment functional on $\Pi^d$ which is known as a {\it Uvarov modification} of $u$.

If we def\/ine the diagonal matrix $\Lambda = \operatorname{diag}\{\lambda_1,\ldots,\lambda_N\}$, then, for $p, q \in \Pi^d$, we can write
\begin{gather} \label{sum-mass}
\langle v, p q \rangle = \langle u, p q \rangle + (p(\xi_1), p(\xi_2), \ldots, p(\xi_N)) \Lambda
\begin{pmatrix}
q(\xi_1)\\ q(\xi_2)\\ \vdots\\ q(\xi_N)
\end{pmatrix} = \langle u, p q \rangle + \mathsf{p}(\xi)^t \Lambda \mathsf{q}(\xi),
\end{gather}
where $\mathsf{p}(\xi) = (p(\xi_1), p(\xi_2), \ldots, p(\xi_N))^t$ and $\mathsf{q}(\xi) = (q(\xi_1), q(\xi_2), \ldots, q(\xi_N))^t$, for all $p, q\in\Pi^d$.

If we generalize equation \eqref{sum-mass} using a non-diagonal matrix $\Lambda$, then $v$ is not a~moment functional but it can be seen as a~bilinear form.

If the moment functional $v$ is quasi-def\/inite, our f\/irst result shows that orthogonal polynomials with respect to $v$ can be derived in terms of those with respect to $u$. To simplify the proof of this result we will make use of a vector-matrix notation which we will introduce next.

Throughout this section, we shall f\/ix $\{\mathbb{P}_n\}_{n\geqslant0}$ as an orthogonal polynomial system associated with $u$. We denote by $\mathsf{P}_n(\xi)$ the matrix whose columns are $\mathbb{P}_n(\xi_i)$
\begin{gather} \label{sP}
\mathsf{P}_n(\xi)= \left(\mathbb{P}_n(\xi_1) | \mathbb{P}_n(\xi_2) | \dots | \mathbb{P}_n(\xi_N) \right) \in \mathcal{M}_{r_n^d \times N}(\mathbb{R}),
\end{gather}
denote by $\mathcal{K}_{n}$ the matrix whose entries are the kernels $K_{n}(u;\xi_i,\xi_j)$ def\/ined in~\eqref{kernelK},
\begin{gather} \label{cK}
\mathcal{K}_{n} = \big(K_{n}(u;\xi_i,\xi_j) \big)_{i,j=1}^N \in \mathcal{M}_{N\times N}(\mathbb{R}),
\end{gather}
and, f\/inally, denote by $\mathsf{K}_{n}(\xi,\mathbf{x})$ the vector of polynomials
\begin{gather} \label{sK}
 {\mathsf{K}_{n}(\xi,\mathbf{x})} = \big(K_{n}(u; \xi_1,\mathbf{x}), K_{n}(u;\xi_2,\mathbf{x}), \ldots, K_{n}(u;\xi_N,\mathbf{x})\big)^{t}.
\end{gather}
From the fact that $K_n(u;\mathbf{x},\mathbf{y}) - K_{n-1}(u;\mathbf{x},\mathbf{y}) = P_n(u;\mathbf{x},\mathbf{y}) = \mathbb{P}_n (\mathbf{x})^{t} H_n^{-1} \mathbb{P}_n (\mathbf{y})$, we have immediately the following relations
\begin{gather}
\mathsf{P}_n^{t}(\xi) H_n^{-1} \mathbb{P}_n (\mathbf{x}) = \mathsf{K}_n(\xi,\mathbf{x}) - \mathsf{K}_{n-1}(\xi,\mathbf{x}), \label{P-K1}\\
\mathsf{P}_n^{t}(\xi) H_n^{-1} \mathsf{P}_n (\xi) = \mathcal{K}_n - \mathcal{K}_{n-1}, \label{P-K2}
\end{gather}
which will be used below.

Now we are ready to state and prove the main result in this section. In fact, we give a~necessary and suf\/f\/icient condition in order to ensure the quasi-def\/initeness of the modif\/ied moment functional in terms of the non-singularity of a matrix.

\begin{Theorem} \label{main-theorem}
Let $u$ be a quasi-definite moment functional and assume that the moment functional $v$ defined in~\eqref{v-def} is quasi-definite. Then the matrices
\begin{gather} \label{I-Lambda-K}
I_N + \Lambda {\mathcal{K}_{n-1}}
\end{gather}
are invertible for $n = 1, 2, \ldots$, and any polynomial system $\{\mathbb{Q}_n\}_{n\geqslant 0}$ orthogonal with respect to $v$ can be written in the form
\begin{gather}
\mathbb{Q}_0(\mathbf{x})= \mathbb{P}_0(\mathbf{x}),\nonumber\\
\mathbb{Q}_n(\mathbf{x})= \mathbb{P}_n(\mathbf{x}) - {\mathsf{P}_n(\xi)} (I_N + \Lambda {\mathcal{K}_{n-1}})^{-1} \Lambda
{\mathsf{K}_{n-1}(\xi,\mathbf{x})}, \qquad n\geqslant 1,\label{ex-expl}
\end{gather}
where $\{\mathbb{P}_n\}_{n\geqslant 0}$ is a polynomial system orthogonal with respect to $u$. Moreover, the invertible matrices $\hat{H}_n = \langle v, \mathbb{Q}_n \mathbb{Q}^{t}_n \rangle$ satisfy
\begin{gather} \label{direct}
 \hat{H}_n = H_n + \mathsf{P}_n(\xi) (I_N + \Lambda \mathcal{K}_{n-1})^{-1} \Lambda
 \mathsf{P}_n^{t}(\xi).
\end{gather}

Conversely, if the matrices defined in \eqref{I-Lambda-K} and \eqref{direct} are invertible then the polynomial system $\{\mathbb{Q}_n\}_{n\geqslant 0}$ defined by \eqref{ex-expl} constitutes an orthogonal polynomial system with respect~to $v$, and therefore $v$ is quasi-definite.
\end{Theorem}

\begin{proof}
Let us assume that $v$ is a quasi-def\/inite moment functional and let $\{\mathbb{Q}_n\}_{n\geqslant 0}$ be an OPS with respect to $v$. We can select an OPS $\{\mathbb{P}_n\}_{n\geqslant 0}$ with respect to~$u$ such that $\mathbb{Q}_n$ has the same leading coef\/f\/icient as $\mathbb{P}_n$, for $n\geqslant0$, in particular we have $\mathbb{Q}_0 = \mathbb{P}_0$.

From this assumption, the components of $\mathbb{Q}_n-\mathbb{P}_n$ are polynomials in $\Pi_{n-1}^d$ for $n\geqslant 1$, then, we can express them as linear combinations of orthogonal polynomials $\mathbb{P}_0, \mathbb{P}_1, \ldots, \mathbb{P}_{n-1}$. In vector-matrix notation, this means that
\begin{gather*}
\mathbb{Q}_n(\mathbf{x}) = \mathbb{P}_n(\mathbf{x}) + \sum_{j=0}^{n-1} M_j^n \mathbb{P}_j(\mathbf{x}),
\end{gather*}
where $M_j^n$ are matrices of size $r_n^d\times r_j^d$. These coef\/f\/icient matrices can be determined from the orthogonality of $\mathbb{P}_n$ and $\mathbb{Q}_n$. Indeed, $\langle v, \mathbb{Q}_n \mathbb{P}^{t}_j \rangle =0$ for $0 \leqslant j \leqslant n-1$, which shows, by def\/inition of $v$ and the fact that $\mathbb{P}_j$ is orthogonal,
\begin{gather*}
 M_j^n = \langle u, \mathbb{Q}_n \mathbb{P}^{t}_j \rangle H_j^{-1} = - \mathsf{Q}_n(\xi) \Lambda \mathsf{P}^{t}_j(\xi) H_j^{-1},
\end{gather*}
where $\mathsf{P}_j(\xi)$ is def\/ined as in \eqref{sP} and $\mathsf{Q}_n(\xi)= \big(\mathbb{Q}_n(\xi_1) | \mathbb{Q}_n(\xi_2) | \dots | \mathbb{Q}_n(\xi_N)\big)$ is the analogous matrix with $\mathbb{Q}_n(\xi_i)$ as its column vectors. Consequently, we obtain
\begin{gather}
\mathbb{Q}_n(\mathbf{x}) = \mathbb{P}_n(\mathbf{x}) - \sum_{j=0}^{n-1} \mathsf{Q}_n(\xi) \Lambda
 \mathsf{P}^{t}_j(\xi) H_j^{-1} \mathbb{P}_j(\mathbf{x}) = \mathbb{P}_n(\mathbf{x}) - \mathsf{Q}_n(\xi) \Lambda \mathsf{K}_{n-1}(\xi,\mathbf{x}),
 \label{QnPj}
\end{gather}
where the second equation follows from relation \eqref{P-K1}, which leads to a~telescopic sum that sums up to $\mathsf{K}_{n-1}(\xi,\mathbf{x})$. Setting $\mathbf{x}=\xi_i$, we obtain
\begin{gather*}
\mathbb{Q}_n(\xi_i) = \mathbb{P}_n(\xi_i) - \mathsf{Q}_n(\xi) \Lambda \mathsf{K}_{n-1}(\xi,\xi_i), \qquad 1 \leqslant i \leqslant N,
\end{gather*}
which by the def\/inition of $\mathcal{K}_{n-1}$ at \eqref{cK} leads to
\begin{gather*}
 \mathsf{Q}_n(\xi) = \mathsf{P}_n(\xi) - \mathsf{Q}_n(\xi) \Lambda \mathcal{K}_{n-1},
\end{gather*}
and therefore
\begin{gather} \label{pn(c)}
 \mathsf{Q}_n(\xi) (I_N + \Lambda \mathcal{K}_{n-1}) = \mathsf{P}_n(\xi).
\end{gather}
Next, we are going to show that the matrices $I_N + \Lambda \mathcal{K}_{n-1}$ are invertible. Assume that there exists an index $k$ such that $I_N + \Lambda \mathcal{K}_{k-1}$ is non regular, then there exists a vector $C = (c_1, c_2, \ldots, c_N)^{t}$ satisfying $(I_N + \Lambda \mathcal{K}_{k-1}) C = 0$. From \eqref{pn(c)} we deduce $\mathsf{P}_k(\xi) C = 0$, which implies $(I_N + \Lambda \mathcal{K}_{k}) C$ $= 0$ using~\eqref{P-K2},
and therefore we conclude
\begin{gather} \label{disc_moment}
\mathsf{P}_n(\xi) C = 0, \qquad n = k, k+1, \ldots.
\end{gather}
Let us consider the discrete moment functional $\mathcal{L} = c_1 \delta_{\xi_1} + \dots + c_N \delta_{\xi_N}$, then~\eqref{disc_moment} implies that $\mathcal{L}$ vanishes on every orthogonal polynomial of total degree $n \geqslant k$. Using duality we can deduce the existence of a polynomial $q(\mathbf{x})$ of total degree $k-1$ such that $\mathcal{L} = q(\mathbf{x}) u$. Let $p(\mathbf{x})$ be a non zero polynomial vanishing at $\xi_1, \xi_2, \ldots, \xi_N$, then we get $p(\mathbf{x}) q(\mathbf{x}) u = p(\mathbf{x}) \mathcal{L} = 0$, which contradicts the quasi-def\/inite character of $u$.

Now, solving for $\mathsf{Q}_n (\xi)$ in equation \eqref{pn(c)} we get
\begin{gather}\label{qn(c)}
 \mathsf{Q}_n(\xi) = \mathsf{P}_n(\xi) (I_N + \Lambda \mathcal{K}_{n-1})^{-1},
\end{gather}
and substituting this expression into \eqref{QnPj} establishes (\ref{ex-expl}).

Finally, from \eqref{ex-expl} and \eqref{qn(c)} we obtain
\begin{gather*}
 \hat{H}_n = \langle v, \mathbb{Q}_n \mathbb{Q}_n^{t}\rangle = \langle v, \mathbb{Q}_n \mathbb{P}_n^{t}\rangle =\langle u, \mathbb{Q}_n \mathbb{P}_n^{t}\rangle + \mathsf{Q}_n(\xi) \Lambda \mathsf{P}_n^{t}(\xi) \\
\hphantom{\hat{H}_n}{} = H_{n} + \mathsf{P}_n(\xi) (I_N + \Lambda \mathcal{K}_{n-1})^{-1} \Lambda \mathsf{P}_n^{t}(\xi),
\end{gather*}
which proves (\ref{direct}).

Conversely, if we def\/ine polynomials $\mathbb{Q}_n$ as in \eqref{ex-expl}, then above proof shows that $\mathbb{Q}_n$ is orthogonal with respect to~$v$. Since $\mathbb{Q}_n$ and~$\mathbb{P}_n$ have the same leading coef\/f\/icient, it is evident that~$\{\mathbb{Q}_n \}_{n \geqslant 0}$ is an OPS in~$\Pi^d$.
\end{proof}

From now on, let us assume that $v$ is a quasi-def\/inite moment functional and $\{\mathbb{Q}_n\}_{n\geqslant0}$ is an OPS with respect to $v$ as given in \eqref{ex-expl}. Then, the invertible matrix $\hat{H}_n=\langle v, \mathbb{Q}_n \mathbb{Q}^{t}_n \rangle$ can be expressed in terms of matrices involving only $\{\mathbb{P}_n\}_{n \geqslant 0}$, as we have shown in Theorem~\ref{main-theorem}. It turns out that this happens also for $\hat{H}_n^{-1}$.

\begin{Proposition}
In the conditions of Theorem {\rm \ref{main-theorem}}, for $n \geqslant 0$, the following identity holds,
\begin{gather} \label{inversa}
 \hat{H}^{-1}_n = H^{-1}_{n} - H^{-1}_{n} \mathsf{P}_n(\xi) (I_N + \Lambda \mathcal{K}_{n})^{-1} \Lambda \mathsf{P}_n^{t}(\xi) H^{-1}_{n}.
\end{gather}
\end{Proposition}

\begin{proof}Formula \eqref{inversa} is a direct consequence of the Sherman--Morrison--Woodbury identity for the inverse of the perturbation of a non singular matrix (see \cite[p.~51]{GvL}).
\end{proof}

Our next result gives explicit formulas for the reproducing kernels associated with~$v$, which we denote by
\begin{gather*}
P_m(v; \mathbf{x},\mathbf{y}) = \mathbb{Q}^{t}_m(\mathbf{x}) \hat{H}^{-1}_m \mathbb{Q}_m(\mathbf{y}),\\
K_n(v; \mathbf{x},\mathbf{y}) = \sum_{m=0}^n P_m(v;\mathbf{x},\mathbf{y}) = \sum_{m=0}^n\mathbb{Q}^{t}_m(\mathbf{x}) \hat{H}^{-1}_m \mathbb{Q}_m(\mathbf{y}).
\end{gather*}

\begin{Lemma}
Let $u$ be a quasi-definite moment functional, $\mathcal{K}_{m}$ defined by \eqref{cK}, and let $\Lambda = \operatorname{diag}\{\lambda_1,\ldots,\lambda_N\}$, with $\lambda_i\neq 0$, $i=1, 2, \ldots, N$. Then, for $m\geqslant 0$, $(I_N + \Lambda \mathcal{K}_{m})^{-1} \Lambda$ is a symmetric matrix.
\end{Lemma}

\begin{proof}
$(I_N + \Lambda \mathcal{K}_{m})^{-1} \Lambda$ is a symmetric matrix as it is the inverse of the symmetric matrix $\Lambda^{-1} (I_N + \Lambda \mathcal{K}_{m}) = \Lambda^{-1} + \mathcal{K}_{m}$.
\end{proof}

Next theorem establishes a relation between the kernels of both families. Similar tools as those used in the proof of Theorem~2.5 in~\cite{DFPPX10} can be applied to obtain this result.

\begin{Theorem}Suppose that we are in the conditions of Theorem~{\rm \ref{main-theorem}}. Then, for $m\geqslant 0$, we get
\begin{gather*}
P_m(v; \mathbf{x},\mathbf{y}) = P_m(u; \mathbf{x},\mathbf{y}) - \mathsf{K}_m^{t}(\xi,\mathbf{x}) (I_N + \Lambda \mathcal{K}_{m})^{-1} \Lambda
 \mathsf{K}_m(\xi,\mathbf{y}) \\
\hphantom{P_m(v; \mathbf{x},\mathbf{y})=}{} + \mathsf{K}_{m-1}^{t}(\xi,\mathbf{x}) (I_N + \Lambda \mathcal{K}_{m-1})^{-1}
 \Lambda \mathsf{K}_{m-1}(\xi,\mathbf{y}),
\end{gather*}
where we assume $\mathsf{K}_{-1}(\mathbf{x},\mathbf{y}) \equiv 0$. Furthermore, for $n \geqslant 0$,
\begin{gather}\label{kernel}
K_n(v; \mathbf{x},\mathbf{y}) = K_n(u; \mathbf{x},\mathbf{y}) - \mathsf{K}_n^{t}(\xi,\mathbf{x}) (I_N + \Lambda \mathcal{K}_{n})^{-1} \Lambda
\mathsf{K}_n(\xi,\mathbf{y}).
\end{gather}
\end{Theorem}

\section{Christof\/fel modif\/ication}\label{section4}

Christof\/fel modif\/ication of a quasi-def\/inite moment functional $u$ will be studied in this section. We def\/ine the {\it Christoffel modification} of $u$ as the moment functional
\begin{gather*}
v = \lambda(\mathbf{x}) u,
\end{gather*}
acting as follows
\begin{gather*}
\langle v, p(\mathbf{x})\rangle = \langle \lambda(\mathbf{x}) u, p(\mathbf{x}) \rangle = \langle u, \lambda(\mathbf{x}) p(\mathbf{x}) \rangle,
\qquad \forall\, p(\mathbf{x})\in\Pi^d.
\end{gather*}

We will work with the particular case when the polynomial $\lambda(\mathbf{x})$ has total degree~$2$. We must remark that in several variables there exist polynomials of second degree that they can not be factorized as a product of two polynomials of degree~$1$, and then this case is not a trivial extension of the case considered in~\cite{APPR14}. Using a block matrix formalism for the three term relations, this case have been also considered in~\cite{AM2} and~\cite{AM3} for arbitrary degree polynomials.

Let $u$ be a quasi-def\/inite moment functional, and let $\lambda(\mathbf{x})$ be a polynomial in $d$ variables of total degree $2$. In terms of the canonical basis, this polynomial can be written as
\begin{gather}\label{lambda}
\lambda(\mathbf{x}) = \textbf{a}_2 \mathbb{X}_2 +\textbf{a}_1 \mathbb{X}_1+\textbf{a}_0 \mathbb{X}_0= \sum_{i=1}^d \sum_{j=i}^d a^{(2)}_{ij} x_i
x_j + \sum_{i=1}^d a^{(1)}_i x_i +a^{(0)},
\end{gather}
where $\textbf{a}_k\in\mathcal{M}_{1\times r_k^d}(\mathbb{R})$, for $k=0,1,2$, whose explicit expressions are
\begin{gather*}
\textbf{a}_2 = \big(a^{(2)}_{11}, a^{(2)}_{12},\ldots,a^{(2)}_{1d}, a^{(2)}_{22},\ldots,a^{(2)}_{2d},\ldots,a^{(2)}_{dd}\big),\\
\textbf{a}_1= \big(a^{(1)}_1, a^{(1)}_2,\ldots,a^{(1)}_d\big),\\
\textbf{a}_0 = (a^{(0)}),
\end{gather*}
with the conditions $|\textbf{a}_2| = \sum\limits_{i=1}^d \sum\limits_{j=i}^d |a^{(2)}_{ij}| \neq 0$, and
\begin{gather}\label{lambda1}
\langle u, \lambda(\mathbf{x})\rangle\neq 0.
\end{gather}

Observe that the f\/irst moment of $v$ is given by
\begin{gather*}
\hat{\mu}_0 = \langle v,1\rangle=\langle u, \lambda(\mathbf{x})\rangle,
\end{gather*}
and using \eqref{lambda1}, we get $\hat{\mu}_0\neq 0$.

First we are going to describe the relations between the moment matrices of both functionals $u$ and $v$. Taking into account that, for $h\geqslant 0$, we get
\begin{gather*}
x_j \mathbb{X}_h = L_{h,j} \mathbb{X}_{h+1},\qquad x_i x_j \mathbb{X}_h = L_{h,j} L_{h+1,i} \mathbb{X}_{h+2},
\end{gather*}
the $r^d_h\times r^d_k$ block of moments for the functional~$v$, $\hat{\textbf{m}}_{h,k}$, can be expressed in terms of the block of moments for the functional~$u$, $\textbf{m}_{h,k}$, in the following way
\begin{gather*}
\hat{\textbf{m}}_{h,k} = \langle v, \mathbb{X}_h \mathbb{X}_k^t\rangle = \langle \lambda(\mathbf{x}) u, \mathbb{X}_h \mathbb{X}_k^t\rangle =
\langle u, \lambda(\mathbf{x}) \mathbb{X}_h \mathbb{X}_k^t\rangle \\
\hphantom{\hat{\textbf{m}}_{h,k}}{} = a^{(0)} \textbf{m}_{h,k} + A_{h,1} \textbf{m}_{h+1,l} + A_{h,2} \textbf{m}_{h+2,k},
\end{gather*}
where
\begin{gather}\label{A_k}
A_{h,1} = \sum_{i=1}^d a^{(1)}_i L_{h,i}, \qquad A_{h,2} = \sum_{i=1}^d \sum_{j=i}^d a^{(2)}_{ij} L_{h,j} L_{h+1,i}
\end{gather}
are matrices of orders $r_h^d\times r_{h+1}^d$ and $r_h^d\times r_{h+2}^d$, respectively, and $A_{h,2}$ has full rank. If we def\/ine the block matrices
\begin{gather*}
\textbf{A}_{n,1} = \left(\begin{array}{@{}cccc|c@{}}
0 & A_{0,1} & & & 0 \\
 & 0 & A_{1,1} & & 0 \\
 & & \ddots & \ddots & \vdots \\
 & & & 0 & A_{n,1}
\end{array}\right),\qquad
\textbf{L}_{n} = \left(\begin{array}{@{}cccc|c@{}}
I_{r^d_0} & & & & 0 \\
 & I_{r^d_1} & & & 0 \\
 & & \ddots & & \vdots\\
 & & & I_{r^d_n} & 0
\end{array}\right),
\end{gather*}
both of dimension $\textbf{r}^d_n\times \textbf{r}^d_{n+1}$, and
\begin{gather*}
\textbf{A}_{n,2} = \left(\begin{array}{@{}cccccc|c@{}}
0 & 0 & A_{0,2} & & & & 0 \\
 & 0 & 0 & A_{1,2} & & & 0 \\
 & & 0 & 0 & A_{2,2} & & 0 \\
 & & & \ddots & \ddots & \ddots & \vdots \\
 & & & & 0 & 0 & A_{n,2}
\end{array}\right)
\end{gather*}
of dimension $\textbf{r}^d_n\times \textbf{r}^d_{n+2}$, then we can write the {\it moment matrix} for the functional~$v$ as the following perturbation of the original moment matrix
\begin{gather*}
\hat{\textbf{M}}_n = \textbf{a}_0 \textbf{M}_n +\textbf{A}_{n,1} \textbf{M}_{n+1} \textbf{L}_{n}^t + \textbf{A}_{n,2} \textbf{M}_{n+2} \textbf{L}_{n+1}^t \textbf{L}_{n}^t.
\end{gather*}

If $u$ and $v$ are quasi-def\/inite, we want to relate both orthogonal polynomial systems $\{\mathbb{P}_n \}_{n \geqslant 0}$ and $\{\mathbb{Q}_n \}_{n \geqslant 0}$ associated with $u$ and $v$, respectively. As usual in this paper, for $n\geqslant 0$, we denote $H_n=\langle u, \mathbb{P}_n \mathbb{P}_n^t\rangle$, and $\hat{H}_n=\langle u, \mathbb{Q}_n \mathbb{Q}_n^t\rangle$, both symmetric and invertible matrices.

\begin{Theorem}\label{teo2}
Let $u$ and $v$ be two quasi-definite moment functionals, and let $\{\mathbb{P}_n \}_{n \geqslant 0}$ and $\{\mathbb{Q}_n \}_{n \geqslant 0}$ be monic OPS associated with $u$ and $v$, respectively. The following statements are equivalent:
\begin{enumerate}\itemsep=0pt
\item[$(1)$] There exists a polynomial $\lambda(\mathbf{x})$ of exact degree two such that
\begin{gather*}
v =\lambda(\mathbf{x}) u.
\end{gather*}

\item[$(2)$] For $n\geqslant 1$, there exist matrices $M_{n} \in {\mathcal{M}}_{r^d_{n}\times r^d_{n-1}}(\mathbb{R})$, $N_{n} \in {\mathcal{M}}_{r^d_{n}\times r^d_{n-2}}(\mathbb{R})$, with $N_{2}\not\equiv0$, such that
\begin{gather}\label{RelOPS}
\mathbb{P}_n = \mathbb{Q}_n + M_{n} \mathbb{Q}_{n-1} + N_{n} \mathbb{Q}_{n-2}.
\end{gather}
\end{enumerate}
\end{Theorem}

\begin{proof} First, we prove (1) $\Rightarrow$ (2). Let us assmue that $v=\lambda(\mathbf{x}) u$ where $\lambda(\mathbf{x})=\textbf{a}_2 \mathbb{X}_2 +\textbf{a}_1 \mathbb{X}_1+\textbf{a}_0 \mathbb{X}_0$, and $|\textbf{a}_2|\neq 0$.

Since $\{\mathbb{Q}_n \}_{n \geqslant 0}$ is a basis of the space of polynomials, and $\mathbb{P}_n$ and $\mathbb{Q}_n$ are monic, then
\begin{gather*}
\mathbb{P}_n=\mathbb{Q}_n + \sum_{j=0}^{n-1} M^n_{j} \mathbb{Q}_{j},
\end{gather*}
where $M^n_{j}\in\mathcal{M}_{r^d_n\times r^d_j}(\mathbb{R})$, and
\begin{gather*}
M_j^n = \langle v, \mathbb{P}_{n} \mathbb{Q}_{j}^t \rangle \hat{H}_j^{-1}, \qquad 0\leqslant j \leqslant n-1.
\end{gather*}
Given that the degree of $\lambda(\mathbf{x})$ is $2$, from the orthogonality of $\mathbb{P}_{n}$ we get
\begin{gather*}
\langle v, \mathbb{P}_{n} \mathbb{Q}_{j}^t \rangle = \langle \lambda(\mathbf{x})u, \mathbb{P}_{n}
\mathbb{Q}_{j}^t \rangle = \langle u, \mathbb{P}_{n} \lambda(\mathbf{x}) \mathbb{Q}_{j}^t \rangle =0,
\qquad \text{for } j<n-2,
\end{gather*}
and then $M_j^n=0$ for $j<n-2$. Therefore, \eqref{RelOPS} holds with $M^n_{n-1} = M_n$ and $M^n_{n-2} = N_n$.

To compute $N_2$, we use $\mathbb{P}_2 = \mathbb{Q}_2 + M_{2} \mathbb{Q}_{1} + N_{2} \mathbb{Q}_{0}$, and thus
\begin{gather*}
\langle v, \mathbb{P}_2 \mathbb{Q}_0^t\rangle = \langle v, (\mathbb{Q}_2 + M_{2} \mathbb{Q}_{1} + N_{2} \mathbb{Q}_{0})\mathbb{Q}_0^t\rangle = N_2 \hat{H}_{0}.
\end{gather*}
On the other hand,
\begin{gather*}
\langle v, \mathbb{P}_2 \mathbb{Q}_0^t\rangle = \langle u, \lambda(\mathbf{x}) \mathbb{P}_2\rangle =
\langle u, \mathbb{P}_2(\mathbb{X}_2^t \textbf{a}^t_2 + \mathbb{X}_1^t \textbf{a}^t_1 +
\mathbb{X}_0^t \textbf{a}^t_0) \rangle = H_2 \textbf{a}^t_2,
\end{gather*}
and therefore
\begin{gather*}
N_2 = H_2 \textbf{a}^t_2 \hat{H}_{0}^{-1}\in \mathcal{M}_{r_2^d\times 1}(\mathbb{R}),
\end{gather*}
has full rank since $|\textbf{a}_2|\neq 0$. So, $N_2\not\equiv 0$ since $N_2$ is a column matrix.

Conversely, we see (2)~$\Rightarrow$~(1). Using the dual basis, and the same reasoning as in the proof of Lemma~1 in~\cite{APPR14}, we obtain
\begin{gather*}
v=\sum_{n=0}^{+\infty} \mathbb{P}_n^t H_n^{-1} E_n u,
\end{gather*}
where $E_n^t = \langle v,\mathbb{P}_n^t \rangle$, $n\geqslant 0$. By \eqref{RelOPS}, we get
\begin{gather*}
E_0^t = \langle v, \mathbb{Q}_0^t \rangle = \hat{H}_0 \neq 0, \\
E_1^t = \langle v, \mathbb{Q}_1^t + \mathbb{Q}_{0}^t M_1^t \rangle =\hat{H}_0 M_1^t, \\
E_2^t = \langle v, \mathbb{Q}_2^t + \mathbb{Q}_{1}^t M_2^t+ \mathbb{Q}_{0}^t N_2^t \rangle = \hat{H}_0 N_2^t,\\
E_n^t = \langle v, \mathbb{Q}_n^t + \mathbb{Q}_{n-1}^t M_{n}^t+ \mathbb{Q}_{n-2}^t N_{n}^t \rangle = 0, \qquad n\geqslant3.
\end{gather*}
Then,
\begin{gather*}
v = \big(\mathbb{P}_2^t H_2^{-1} N_2 \hat{H}_0 + \mathbb{P}_1^t H_1^{-1} M_1 \hat{H}_0 + \mathbb{P}_0^t H_0^{-1} \hat{H}_0\big) u,
\end{gather*}
or equivalently, there exists a polynomial
\begin{gather*}
\lambda(\mathbf{x}) = \hat{H}_0 \big(N_2^t H_2^{-1} \mathbb{P}_2 + M_1^t H_1^{-1} \mathbb{P}_1 + H_0^{-1} \mathbb{P}_0\big),
\end{gather*}
such that $v=\lambda(\mathbf{x}) u$. Since $N_2\not\equiv 0$, then $\lambda(\mathbf{x})$ has exact degree~2.

Moreover, we will prove that $N_n$ has full rank for $n\geqslant 2$. In fact, using \eqref{RelOPS}, we get
\begin{gather*}
\langle v, \mathbb{P}_n \mathbb{Q}_{n-2}^t \rangle = \langle v,
[\mathbb{Q}_n + M_{n} \mathbb{Q}_{n-1} + N_{n} \mathbb{Q}_{n-2}]\mathbb{Q}_{n-2}^t \rangle =
 N_{n} \hat{H}_{n-2}.
\end{gather*}
On the other hand,
\begin{gather*}
\langle v, \mathbb{P}_n \mathbb{Q}_{n-2}^t \rangle = \langle u, \mathbb{P}_n \lambda(\mathbf{x})
\mathbb{Q}_{n-2}^t \rangle = \sum_{i=1}^d \sum_{j=i}^d
a^{(2)}_{ij} \langle u, \mathbb{P}_n x_i x_j \mathbb{Q}_{n-2}^t \rangle\\
\hphantom{\langle v, \mathbb{P}_n \mathbb{Q}_{n-2}^t \rangle}{} = \sum_{i=1}^d \sum_{j=i}^d
a^{(2)}_{ij} \langle u, \mathbb{P}_n x_i x_j \mathbb{X}_{n-2}^t \rangle = \sum_{i=1}^d \sum_{j=i}^d
a^{(2)}_{ij} \langle u, \mathbb{P}_n \mathbb{X}_{n}^t \rangle L_{n-1,i}^t L_{n-2,j}^t \\
\hphantom{\langle v, \mathbb{P}_n \mathbb{Q}_{n-2}^t \rangle}{} = \sum_{i=1}^d \sum_{j=i}^d a^{(2)}_{ij} H_n L_{n-1,i}^t L_{n-2,j}^t = H_n A_{n-2,2}^t,
\end{gather*}
where $A_{n-2,2}$ was def\/ined in \eqref{A_k}. Then,
\begin{gather}\label{Ndef}
N_n \hat{H}_{n-2} = H_n A_{n-2,2}^t, \qquad n\geqslant 2.
\end{gather}
Therefore $N_n$ is full rank for $n \geqslant 2$ since $H_n$ and $\hat{H}_{n-2}$ are invertible matrices, and the rank of a matrix is invariant by multiplication times non-singular matrices \cite[p.~13]{HJ85}.
\end{proof}

\begin{Remark} When both moment functionals are quasi-def\/inite, that is, when both OPS $\{\mathbb{P}_n\}_{n\geqslant 0}$ and $\{\mathbb{Q}_n\}_{n\geqslant 0}$ exist, the orthogonality condition of the second family with respect to $v = \lambda(\mathbf{x}) u$ trivially implies that the polynomial entries in $\lambda(\mathbf{x}) \mathbb{Q}_n$ are quasi-orthogonal with respect to the f\/irst moment functional~$u$. In fact, there exist matrices of adequate size such that
\begin{gather*}
\lambda(\mathbf{x}) \mathbb{Q}_n = \sum_{k=0}^{n+\deg\lambda(\mathbf{x})} A_k^n \mathbb{P}_k,
\end{gather*}
where
\begin{gather*}
A_k^n H_k = \langle u, \lambda(\mathbf{x}) \mathbb{Q}_n \mathbb{P}_k\rangle = \langle \lambda(\mathbf{x}) u, \mathbb{Q}_n \mathbb{P}_k\rangle =
\langle v, \mathbb{Q}_n \mathbb{P}_k\rangle.
\end{gather*}
Then, $A_k^n =0$, for $0\leqslant k \leqslant n-1$, and therefore
\begin{gather*}
\lambda(\mathbf{x}) \mathbb{Q}_n = \sum_{k=n}^{n+\deg\lambda(\mathbf{x})} A_k^n \mathbb{P}_k.
\end{gather*}
The matrix version of this relation is the f\/irst identity in Proposition~2.7 of~\cite{AM2}.
\end{Remark}

Now, we assume that $u$ is a quasi-def\/inite moment functional, and $\{\mathbb{P}_n \}_{n \geqslant 0}$ is the monic OPS associated with $u$. Then $\{\mathbb{P}_n \}_{n \geqslant 0}$ satisfy the three term relations \eqref{3TR} with the rank condi\-tions~\eqref{cond_rank_1},~\eqref{cond_rank_2}. Def\/ining recursively the monic polynomial system $\{\mathbb{Q}_n \}_{n \geqslant 0}$ by means of~\eqref{RelOPS}, we want to deduce its relation with $u$ as well as conditions for its quasi-def\/initeness.

\begin{Theorem}\label{T42}
Let $\{\mathbb{P}_n \}_{n \geqslant 0}$ be a monic OPS associated with the quasi-definite moment functional $u$, and let $\{M_n\}_{n \geqslant 1}$ and $\{N_n\}_{n \geqslant 2}$ be two sequences of matrices of orders $r_n^d\times r_{n-1}^d$ and $r_n^d\times r_{n-2}^d$ respectively, such that $N_2\not\equiv 0$. Define recursively the monic polynomial system
\begin{gather*}
\mathbb{Q}_0 = \mathbb{P}_0,\\
\mathbb{Q}_1 = \mathbb{P}_1 - M_1 \mathbb{P}_{0},\\
\mathbb{Q}_n = \mathbb{P}_n - M_n \mathbb{Q}_{n-1} - N_n \mathbb{Q}_{n-2}, \qquad n\geqslant 2.
\end{gather*}
Then $\{\mathbb{Q}_n \}_{n \geqslant 0}$ is a monic OPS associated with a quasi-definite moment functional~$v$, satisfying the three term relation
\begin{gather}\label{RR3TQ}
x_i \mathbb{Q}_n(\mathbf{x}) = L_{n,i} \mathbb{Q}_{n+1}(\mathbf{x}) + \hat{B}_{n,i} \mathbb{Q}_n(\mathbf{x}) + \hat{C}_{n,i} \mathbb{Q}_{n-1}(\mathbf{x}),\qquad 1 \leqslant i \leqslant d,
\end{gather}
with initial conditions $\mathbb{Q}_{-1}(\mathbf{x}) = 0$, $\hat{C}_{-1,i} = 0$, if and only if
\begin{gather}
\hat{B}_{n,i} = B_{n,i} - M_n L_{n-1,i} + L_{n,i} M_{n+1}, \label{RR1}\\
\hat{C}_{n,i} = C_{n,i} - M_n \hat{B}_{n-1,i} + B_{n,i} M_n - N_n L_{n-2,i} + L_{n,i} N_{n+1}, \label{RR2}\\
M_n \hat{C}_{n-1,i} + N_n \hat{B}_{n-2,i} = C_{n,i} M_{n-1} + B_{n,i} N_n,\label{RR3}\\
C_{n,i} N_{n-1} = N_n \hat{C}_{n-2,i}. \label{RR3b}
\end{gather}
In such a case, there exists a polynomial of exact degree two given by
\begin{gather*}
\lambda(\mathbf{x}) = N_2^t H_2^{-1}\mathbb{P}_2 + M_1^t H_1^{-1} \mathbb{P}_1 + H_0^{-1} \mathbb{P}_0,
\end{gather*}
satisfying
\begin{gather*}
v = \lambda(\mathbf{x}) u.
\end{gather*}
\end{Theorem}

\begin{proof}
Replacing \eqref{RelOPS} in \eqref{3TR} we get
\begin{gather*}
x_i(\mathbb{Q}_n + M_n \mathbb{Q}_{n-1} + N_n \mathbb{Q}_{n-2}) = L_{n,i} \mathbb{Q}_{n+1} + (L_{n,i} M_{n+1} + B_{n,i}) \mathbb{Q}_{n} \\
\hphantom{x_i(\mathbb{Q}_n + M_n \mathbb{Q}_{n-1} + N_n \mathbb{Q}_{n-2}) =}{} + (L_{n,i} N_{n+1} + B_{n,i} M_n + C_{n,i}) \mathbb{Q}_{n-1} \\
\hphantom{x_i(\mathbb{Q}_n + M_n \mathbb{Q}_{n-1} + N_n \mathbb{Q}_{n-2}) =}{}
 + (B_{n,i} N_n + C_{n,i} M_{n-1})\mathbb{Q}_{n-2} +C_{n,i} N_{n-1} \mathbb{Q}_{n-3}.
\end{gather*}
On the other hand, if $\{\mathbb{Q}_n \}_{n \geqslant 0}$ satisfy \eqref{RR3TQ} then
\begin{gather*}
x_i(\mathbb{Q}_n + M_n \mathbb{Q}_{n-1} + N_n \mathbb{Q}_{n-2}) =
L_{n,i} \mathbb{Q}_{n+1} + (M_n L_{n-1,i} + \hat{B}_{n,i})\mathbb{Q}_{n}\\
\hphantom{x_i(\mathbb{Q}_n + M_n \mathbb{Q}_{n-1} + N_n \mathbb{Q}_{n-2}) =}{} +(\hat{C}_{n,i} + M_n \hat{B}_{n-1,i} + N_n L_{n-2,i}) \mathbb{Q}_{n-1}\\
 \hphantom{x_i(\mathbb{Q}_n + M_n \mathbb{Q}_{n-1} + N_n \mathbb{Q}_{n-2}) =}{} + (M_n \hat{C}_{n-1,i} + N_n \hat{B}_{n-2,i}) \mathbb{Q}_{n-2} + N_n\hat{C}_{n-2,i} \mathbb{Q}_{n-3}.
\end{gather*}
Subtracting both expressions we get \eqref{RR1}--\eqref{RR3b}.

Conversely, we will prove the three term relation for $\{\mathbb{Q}_n \}_{n \geqslant 0}$ using induction. In fact, for $n=1$, multiplying relation \eqref{RelOPS} times $L_{0,i}$, we get
\begin{gather*}
x_i \mathbb{Q}_0 = L_{0,i} \mathbb{Q}_1 + (B_{0,i}+ L_{0,i} M_1) \mathbb{Q}_0.
\end{gather*}
Let us suppose that \eqref{RR3TQ} is satisf\/ied for $n-1$. Multiplying \eqref{RelOPS} for $n+1$ times $L_{n,i}$ and applying the three term relation for $\{\mathbb{P}_n\}_{n \geqslant 0}$, we get
\begin{gather*}
x_i \mathbb{P}_n - B_{n,i} \mathbb{P}_n -C_{n,i} \mathbb{P}_{n-1} = L_{n,i} \mathbb{Q}_{n+1} + L_{n,i} M_{n+1} \mathbb{Q}_{n} + L_{n,i} N_{n+1} \mathbb{Q}_{n-1}.
\end{gather*}
Replacing again \eqref{RelOPS} in the left hand side, using induction hypotheses and relations \eqref{RR1}--\eqref{RR3b} we get the announced three term relations for $\{\mathbb{Q}_n \}_{n \geqslant 0}$.

Now, we def\/ine the moment functional $v$ as
\begin{gather*}
\langle v, 1\rangle = 1, \qquad \langle v, \mathbb{Q}_n\rangle = 0, \qquad n\geqslant 1.
\end{gather*}
Since $\{\mathbb{Q}_n \}_{n \geqslant 0}$ is a basis of $\Pi^d$, then $v$ is well def\/ined. Following \cite[p.~74]{DX14}, since~$L_{n,i}$ and~$L_n$ have full rank, then
\begin{gather*}
\langle v, \mathbb{Q}_n \mathbb{Q}^t_m\rangle =0, \qquad n\neq m.
\end{gather*}
Finally, we need to prove that $\hat{H}_n = \langle v, \mathbb{Q}_n \mathbb{Q}^t_n\rangle$, is an invertible matrix, for $n\geqslant 0$.

From the def\/inition of $v$, $\hat{H}_0 = \langle v, \mathbb{Q}_0 \mathbb{Q}^t_0\rangle = \langle v, 1\rangle =1$. On the other hand, since expression~\eqref{Ndef} still holds in this case,
\begin{gather*}
N_n \hat{H}_{n-2} = H_n A_{n-2,2}^t, \qquad n\geqslant 2,
\end{gather*}
using the properties of the rank of a product of matrices, we get
\begin{gather*}
\operatorname{rank}\big(N_n \hat{H}_{n-2}\big) = \operatorname{rank}\big(H_n A_{n-2,2}^t\big) = \operatorname{rank} A_{n-2,2}^t = r_{n-2}^d.
\end{gather*}
Therefore
\begin{gather*}r_{n-2}^d = \operatorname{rank}\big(N_n \hat{H}_{n-2}\big)\leqslant \min\big\{\operatorname{rank}N_n, \operatorname{rank}\hat{H}_{n-2}\big\} \leqslant \operatorname{rank}\hat{H}_{n-2} \leqslant r_{n-2}^d,
\end{gather*}
and then $\operatorname{rank}\hat{H}_{n-2} = r^d_{n-2}$. In this way, the moment functional $v$ is quasi-def\/inite and $\{\mathbb{Q}_n \}_{n \geqslant 0}$ is a monic OPS associated with~$v$. Using Theorem~\ref{teo2}, both moment functionals are related by means of a~Christof\/fel modif\/ication
\begin{gather*}
v = \lambda(\mathbf{x}) u,
\end{gather*}
where $ \lambda(\mathbf{x}) = N_2^t H_2^{-1} \mathbb{P}_2 + M_1^t H_1^{-1} \mathbb{P}_1 + H_0^{-1} \mathbb{P}_0$.
\end{proof}

\begin{Remark} Observe that relation \eqref{RR3b} always holds when $\{\mathbb{Q}_n \}_{n \geqslant 0}$ is an OPS. In fact, using~\eqref{C1}, we get
\begin{gather*}
C_{n,i} H_{n-1} = H_n L^t_{n-1,i} \qquad \text{and} \qquad \hat{C}_{n,i} \hat{H}_{n-1} = \hat{H}_n L^t_{n-1,i},
\end{gather*}
and jointly with \eqref{Ndef}, it follows
\begin{gather*}
C_{n,i} N_{n-1} - N_n \hat{C}_{n-2,i} = H_n\big[L^t_{n-1,i} A^t_{n-3,2} - A^t_{n-2,2} L^t_{n-3,i}\big]\hat{H}^{-1}_{n-3}.
\end{gather*}
On the other hand, from \eqref{A_k},
\begin{gather*}
L_{n-3,i} A_{n-2,2} - A_{n-3,2} L_{n-1,i} = L_{n-3,i} \left(\sum_{k=1}^d\sum_{j=k}^d a^{(2)}_{kj} L_{n-2,j}L_{n-1,k}\right) \\
\hphantom{L_{n-3,i} A_{n-2,2} - A_{n-3,2} L_{n-1,i} =}{} - \left(\sum_{k=1}^d\sum_{j=k}^d a^{(2)}_{kj} L_{n-3,j} L_{n-2,k}\right)L_{n-1,i}= 0,
\end{gather*}
from property \eqref{L-i-j}. However, if $\{\mathbb{Q}_n \}_{n \geqslant 0}$ is not orthogonal, then we can not assume {\it a priori} that $\hat{H}_{n-3}$ is non singular, and so \eqref{RR3b} does not necessarily hold.
\end{Remark}

\begin{Remark} In the case when both functionals $u$ and $v$ are quasi-def\/inite, Theorem~\ref{T42} can be rewritten by using a matrix formalism, as it is done in \cite{AM2, AM3}. For $1\leqslant i\leqslant d$, we denote by
\begin{gather*}
\mathcal{J}_i = \begin{pmatrix}
B_{0,i} & L_{0,i} & & & \bigcirc \\
C_{1,i} & B_{1,i} & L_{1,i} & & \\
 & C_{2,i} & B_{2,i} & \ddots & \\
 & & \ddots & \ddots & \\
\bigcirc & & & &
\end{pmatrix}, \qquad
\hat{\mathcal{J}}_i = \begin{pmatrix}
\hat{B}_{0,i} & L_{0,i} & & & \bigcirc \\
\hat{C}_{1,i} & \hat{B}_{1,i} & L_{1,i} & & \\
 & \hat{C}_{2,i} & \hat{B}_{2,i} & \ddots & \\
 & & \ddots & \ddots & \\
\bigcirc & & & &
\end{pmatrix},
\end{gather*}
the respective \emph{block Jacobi matrices} associated with the three term relations \cite[p.~82]{DX14}. Also we def\/ine the \emph{lower triangular block matrix} with identity matrices as diagonal blocks
\begin{gather*}
\mathcal{M} =
\begin{pmatrix}
I & & & & &\bigcirc \\
M_1 & I & & & & \\
N_2 & M_2 & I & & & \\
 & N_3 & M_3 & I & & \\
 & & \ddots & \ddots & \ddots & \\
\bigcirc & & & & &
\end{pmatrix},
\end{gather*}
where $M_n$, $N_n$ are def\/ined in Theorem~\ref{teo2}. Then, formulas \eqref{RR1}--\eqref{RR3b} can be expressed as the matrix product
\begin{gather*}
\mathcal{J}_i \mathcal{M} = \mathcal{M} \hat{\mathcal{J}}_i, \qquad i=1, 2, \ldots, d.
\end{gather*}
The explicit expressions of the matrices $M_n$ and $N_n$ given in Theorem \ref{teo2} lead to matrix relations of Proposition~2.4 in \cite{AM2}.
\end{Remark}

\subsection{Centrally symmetric functionals}

Following \cite[p.~76]{DX14}, a moment functional $u$ is called \emph{centrally symmetric} if it satisf\/ies
\begin{gather*}
\langle u, x^\nu\rangle = 0, \qquad \nu \in \mathbb{N}^d, \qquad |\nu| \quad \textrm{is an odd integer}.
\end{gather*}
This def\/inition constitutes the multivariate extension of the symmetry for a moment functional.

Quasi-def\/inite centrally symmetric moment functionals can be characterized in terms of the matrix coef\/f\/icients of the three term relations~\eqref{3TR}. In fact, $u$ is centrally symmetric if and only if $B_{n,i} = 0$ for all $n \geqslant0$ and $1\leqslant i \leqslant d$.

As a consequence, an orthogonal polynomial of degree $n$ with respect to $u$ is a sum of monomials of even degree if $n$ is even and a sum of monomials of odd degree if $n$ is odd.

Let us suppose that $u$ is a quasi-def\/inite centrally symmetric moment functional, and we def\/ine its Christof\/fel modif\/ication by
\begin{gather*} v = \lambda(\mathbf{x}) u,
\end{gather*}
where $\lambda(\mathbf{x})$ is a polynomial of second degree as \eqref{lambda}. Then

\begin{Proposition}
$v$ is centrally symmetric if and only if $\textbf{a}_1=0$, that is,
\begin{gather*}
\lambda(\mathbf{x}) = \textbf{a}_2 \mathbb{X}_2 +\textbf{a}_0 \mathbb{X}_0= \sum_{i=1}^d \sum_{j=i}^d a^{(2)}_{ij} x_i x_j + a^{(0)}.
\end{gather*}
\end{Proposition}

If $v$ is quasi-def\/inite and centrally symmetric, then relation \eqref{RelOPS} is given by
\begin{gather}
\mathbb{P}_0 = \mathbb{Q}_0,\qquad \mathbb{P}_1 = \mathbb{Q}_1,\qquad \mathbb{P}_n = \mathbb{Q}_n + N_{n} \mathbb{Q}_{n-2}, \qquad n\ge2,\label{sim}
\end{gather}
where $N_n = H_n A_{n-2,2}^t \hat{H}_{n-2}^{-1}$, $n\geqslant 2$.

\section{Examples}\label{section5}

\subsection{Two modif\/ications on the ball}

Let us denote by
\begin{gather*}
\mathbb{B}^d =\big\{\mathbf{x}\in \mathbb{R}^d\colon \|\mathbf{x}\| \leqslant 1\big\} \qquad \textrm{and}
\qquad \mathbb{S}^{d-1}= \big\{\xi\in \mathbb{R}^d\colon \|\xi\| = 1\big\},
\end{gather*}
the unit ball and the unit sphere on $\mathbb{R}^d$, respectively, where
\begin{gather*}
\|\mathbf{x}\| = \sqrt{x_1^2 + x_2^2 + \cdots + x_d^2},
\end{gather*}
denotes the usual Euclidean norm. Consider the weight function
\begin{gather*}
W_\mu(\mathbf{x}) = \big(1-\|\mathbf{x}\|^2\big)^{\mu-1/2},\qquad \mu > -1/2, \qquad \mathbf{x} \in \mathbb{B}^d.
\end{gather*}
Associated with $W_\mu(\mathbf{x})$, we def\/ine the usual inner product on the unit ball
\begin{gather*}
\langle f,g\rangle_\mu = \omega_\mu \int_{\mathbb{B}^d} f(\mathbf{x}) g(\mathbf{x}) W_\mu(\mathbf{x}) d\mathbf{x},
\end{gather*}
where the normalizing constant
\begin{gather*}
\omega_\mu = \left[\int_{\mathbb{B}^d} W_\mu(\mathbf{x}) d\mathbf{x}\right]^{-1} = \frac{\Gamma(\mu + (d+1)/2)}{\pi^{d/2}\Gamma(\mu+1/2)}
\end{gather*}
is chosen in order to have $\langle 1,1\rangle_\mu = 1$.

The associated moment functional is def\/ined by means of its moments
\begin{gather*}
\langle u_\mu, \mathbf{x}^\nu\rangle = \omega_\mu \int_{\mathbb{B}^d} \mathbf{x}^\nu
W_\mu(\mathbf{x}) d\mathbf{x} = \omega_\mu \int_{\mathbb{B}^d} \mathbf{x}^\nu \big(1-\|\mathbf{x}\|^2\big)^{\mu-1/2} d\mathbf{x}.
\end{gather*}
Observe that $u_\mu$ is a centrally symmetric positive-def\/inite moment functional, that is,
\begin{gather*}
\langle u_\mu, \mathbf{x}^\nu\rangle =0, \qquad \text{whenever} \quad |\nu| \quad \text{is an odd integer}.
\end{gather*}

Let us denote by $\{\mathbb{P}^{(\mu)}_n\}_{n\geqslant0}$ a ball OPS. In this case, several explicit bases are known, and we will describe one of them. An orthogonal basis in terms of classical Jacobi polynomials and spherical harmonics is presented in~\cite{DX14}.

\textit{Harmonic polynomials} (see \cite[p.~114]{DX14}) are homogeneous polynomials in $d$ variables $Y(\mathbf{x})$ satisfying the Laplace equation
\begin{gather*}
\Delta Y = 0.
\end{gather*}
Let $\mathcal{H}_n^d$ denote the space of harmonic polynomials of degree $n$. It is well known that
\begin{gather*}
\sigma_n = \dim \mathcal{H}_n^d = \binom{n+d-1}{n} - \binom{n+d-3}{n}, \qquad n\geqslant 3,
\end{gather*}
and $\sigma_0=\dim \mathcal{H}_0^d = 1$, $\sigma_1=\dim \mathcal{H}_1^d = d$.

The restriction of $Y(\mathbf{x}) \in \mathcal{H}_n^d$ to $\mathbb{S}^{d-1}$ is called a \textit{spherical harmonic}. We will use spherical polar coordinates $\mathbf{x}=r \xi$, for $\mathbf{x}\in \mathbb{R}^d$, $r\geqslant 0$, and $\xi\in \mathbb{S}^{d-1}$. If $Y \in \mathcal{H}_n^d$, we use the notation $Y(\mathbf{x})$ to denote the harmonic polynomial, and $Y(\xi)$ to denote the spherical harmonic. This notation is coherent with the fact that if $\mathbf{x} = r \xi$ then $Y(\mathbf{x}) = r^n Y(\xi)$. Moreover, spherical harmonics of dif\/ferent degree are orthogonal with respect to the surface measure on $\mathbb{S}^{d-1}$, and we can choose an orthonormal basis.

Then, an orthonormal basis for the classical ball inner product (see \cite[p.~142]{DX14}) is given by the polynomials
\begin{gather}\label{baseP}
P_{j,k}^{n}(\mathbf{x}; \mu) = h_{j,n}^{-1} P_{j}^{(\mu-\frac{1}{2}, n-2j + \frac{d-2}{2})}\big(2 \|\mathbf{x}\|^2 -1\big) Y_k^{n-2j}(\mathbf{x})
\end{gather}
for $0\leqslant j \leqslant n/2$ and $1\leqslant k\leqslant \sigma_{n-2j}$. Here $P_{j}^{(\alpha, \beta)}(t)$ denotes the $j$-th classical Jacobi polynomial, $\{Y_k^{n-2j}\colon 1\leqslant k\leqslant \sigma_{n-2j}\}$ is an orthonormal basis for $\mathcal{H}_{n-2j}^d$ and the constants $h_{j,n}$ are def\/ined by
\begin{gather*}
[h_{j,n}]^2 = [h_{j,n}(\mu)]^2=\frac{\big(\mu+\frac{1}{2}\big)_{j}\big(\frac{d}{2}\big)_{n-j} \big(n-j+\mu+\frac{d-1}{2}\big)}
{j! \big(\mu+\frac{d+1}{2}\big)_{n-j}\big(n+\mu+\frac{d-1}{2}\big)},
\end{gather*}
where $(a)_m = a (a+1)\cdots(a+m-1)$ denotes the Pochhammer symbol.

This basis can be written as an orthonormal polynomial system in the form
\begin{gather*}
\mathbb{P}_n^{(\mu)} = \big(P^n_{[\frac{n}{2}],\sigma_{n-2[\frac{n}{2}]}}, \ldots, P^n_{[\frac{n}{2}],1}; \dots; P^n_{1,\sigma_{n-2}}, \ldots, P^n_{1,2}, P^n_{1,1}; P^n_{0,\sigma_n}, \ldots, P^n_{0,2}, P^n_{0,1}\big)^t,
\end{gather*}
where we order the entries by reverse lexicographical order of their indexes. Observe that for~$n$ odd, then $n-2[\frac{n}{2}] = 1$, and $\sigma_1 = d$, while for $n$ even, $n-2[\frac{n}{2}] = 0$, and $\sigma_0 = 1$.

\subsubsection{Connection properties for adjacent families of ball polynomials}

Christof\/fel modif\/ication allows us to relate two families of ball polynomials for two values of the parameter $\mu$ dif\/fering by one unity. In fact,
\begin{gather*}
W_{\mu+1}(\mathbf{x}) = \big(1-\|\mathbf{x}\|^2\big)^{\mu+1} = \big(1-\|\mathbf{x}\|^2\big) W_\mu(\mathbf{x}).
\end{gather*}
Then, if we def\/ine $\lambda(\mathbf{x}) = 1-\|\mathbf{x}\|^2 = 1-x_1^2 - x_2^2 - \cdots - x_d^2$, using the matrix formalism \eqref{sim}, we get the following relation
\begin{gather*}
\mathbb{P}_n^{(\mu)}(\mathbf{x}) = F_n \mathbb{P}_n^{(\mu+1)}(\mathbf{x}) +
N_n \mathbb{P}_{n-2}^{(\mu+1)}(\mathbf{x}),
\end{gather*}
where $F_n$ is the non-singular matrix needed to change the leading coef\/f\/icients, and $N_n$ has full rank. However, in this case, we can explicitly give both matrices using the previously described basis.

Here we use the relation for Jacobi polynomials (formula~(22.7.18) in~\cite{AS72})
\begin{gather*}
P_n^{(\alpha,\beta)}(t) = \frac{n+\alpha+\beta+1}{2 n + \alpha + \beta+1} P_n^{(\alpha+1,\beta)}(t)
- \frac{n+\beta}{2 n + \alpha + \beta+1} P_{n-1}^{(\alpha+1,\beta)}(t),
\end{gather*}
for $n\geqslant 0$, in \eqref{baseP} and we can relate ball polynomials corresponding to adjacent families
\begin{gather*}
P_{j,k}^{n}(\mathbf{x}; \mu) = a^n_j P_{j,k}^{n}(\mathbf{x}; \mu+1) - b^n_j P_{j-1,k}^{n-2}(\mathbf{x}; \mu+1),
\end{gather*}
where
\begin{gather*}
a^n_j = \frac{h_{j,n}(\mu+1)}{h_{j,n}(\mu)}\frac{n-j+\mu+(d-1)/2}{n+\mu+(d-1)/2}, \qquad
b^n_j = \frac{h_{j-1,n-2}(\mu+1)}{h_{j,n}(\mu)}\frac{n-j-1+d/2}{n+\mu+(d-1)/2},
\end{gather*}
for $n\geqslant 2$ and $1\leqslant j\leqslant n/2$.

Then, the matrix $F_n$ is a diagonal and invertible matrix given by
\begin{gather*}
F_n = \operatorname{diag} \big(a^n_{[\frac{n}{2}]}, \ldots, a^n_{[\frac{n}{2}]}; \dots; a^n_1, \ldots, a^n_1;a^n_0, \ldots, a^n_0\big),
\end{gather*}
that is, we repeat every $a^n_i$ in the diagonal $\sigma_{n-2i}$ times, for $0\leqslant i \leqslant [\frac{n}{2}]$.

Moreover, the $r^d_n\times r^d_{n-2}$ block matrix $N_n$ has full rank, and
\begin{gather*}
N_n = - L^t_{n-1,1} L^t_{n-2,1} \operatorname{diag}\big(b^n_{[\frac{n}{2}]}, \ldots, b^n_{[\frac{n}{2}]}; \dots; b^n_1, \ldots, b^n_1\big).
\end{gather*}

\subsubsection{Ball--Uvarov polynomials}

Let $u_\mu$ be the classical ball moment functional def\/ined as above, and we add a mass point at the origin and consider the moment functional
\begin{gather*}
\langle v, p(\mathbf{x}) \rangle = \langle u_\mu, p(\mathbf{x}) \rangle + \lambda p(\mathbf{0}).
\end{gather*}
Using Theorem~\ref{main-theorem} we deduce the quasi-def\/inite character of the moment functional $v$ except for a denumerable set of negative values of $\lambda$. Of course, $v$ is positive def\/inite for every positive value of $\lambda$.

Let $\{P_{j,k}^{n}(\mathbf{x})\}$ be the mutually orthogonal basis for the ball polynomials given in \eqref{baseP}. Since spherical harmonics are homogeneous polynomials, we get $Y_k^{n-2j}(\mathbf{0}) = 0$ whenever $n-2j > 0$. Hence, it follows
\begin{gather} \label{pn(0)}
P_{j,k}^{n}(\mathbf{0}) = \begin{cases}
h^{-1}_{{\lfloor\frac{n}{2}\rfloor},n} P_{{\lfloor\frac{n}{2}\rfloor}}^{(\mu-\frac{1}{2},\frac{d-2}{2})}
(-1) &\text{if $n$ is even, $j = {\lfloor\frac{n}{2}\rfloor}$ and $k=1$,}\\
0 &\text{in any other case.}
\end{cases}
\end{gather}

As a consequence we get

\begin{Lemma} For $n \geqslant 0$
\begin{gather}
K_n(u_\mu;\mathbf{x},\mathbf{0}) =
\frac{\big(\mu + \frac{d+1}{2}\big)_{{\lfloor\frac{n}{2}\rfloor}}}{\big(\mu + \frac{1}{2}\big)_{{\lfloor\frac{n}{2}\rfloor}}}
P_{{\lfloor\frac{n}{2}\rfloor}}^{(\frac{d-2}{2},\mu-\frac{1}{2})}\big(1 - 2 \|\mathbf{x}\|^2\big), \nonumber\\
K_n(u_\mu;\mathbf{0},\mathbf{0}) =
\frac{\big(\mu + \frac{d+1}{2}\big)_{{\lfloor\frac{n}{2}\rfloor}}}{\big(\mu + \frac{1}{2}\big)_{{\lfloor\frac{n}{2}\rfloor}}}
\binom{{\lfloor\frac{n}{2}\rfloor} + \frac{d}{2}}{{\lfloor\frac{n}{2}\rfloor}}. \label{Kn0}
\end{gather}
\end{Lemma}

\begin{proof}
Let $h_n^{(\alpha,\beta)}$ be the norm of the Jacobi polynomial $P_n^{(\alpha,\beta)}(t)$ as given in \cite[(4.3.3)]{Sz78}
\begin{gather*}
h_n^{(\alpha,\beta)} = \frac{2^{\alpha+\beta+1}}{2n+\alpha+\beta+1} \frac{\Gamma(n+\alpha+1) \Gamma(n+\beta+1)}{n! \Gamma(n+\alpha+\beta+1)}.
\end{gather*}
Then we have $h^2_{{\lfloor\frac{n}{2}\rfloor},n} = d_{\mu} h_{\lfloor\frac{n}{2}\rfloor}^{(\mu-\frac{1}{2},\frac{d-2}{2})}$, where
\begin{gather*}
d_{\mu} =
\frac{1}{2^{\mu+\frac{d-1}{2}}} \frac{\Gamma\big(\mu+\frac{d+1}{2}\big)}{\Gamma\big(\mu+\frac{1}{2}\big) \Gamma\big(\frac{d}{2}\big)}.
\end{gather*}
From \eqref{pn(0)} we get
\begin{gather*}
K_n(u_\mu;\mathbf{x},\mathbf{0}) = d_{\mu}^{-1} \sum_{j=0}^{{\lfloor\frac{n}{2}\rfloor}}
\big[h_j^{(\mu-\frac{1}{2},\frac{d-2}{2})}\big]^{-1}
P_{j}^{(\mu-\frac{1}{2},\frac{d-2}{2})}\big(2 \|\mathbf{x}\|^2-1\big)
P_{j}^{(\mu-\frac{1}{2},\frac{d-2}{2})}(-1).
\end{gather*}
Here, we can recognize the ${\lfloor\frac{n}{2}\rfloor}$-th kernel of the Jacobi polynomials with parameters $(\mu-\frac{1}{2},\frac{d-2}{2})$, using $P_{j}^{(\alpha,\beta)}(t) = (-1)^j P_{j}^{(\beta, \alpha)}(-t)$ and relation (4.5.3) in \cite{Sz78}, we get
\begin{gather*}
K_n(u_\mu;\mathbf{x},\mathbf{0}) = \frac{\big(\mu + \frac{d+1}{2}\big)_{{\lfloor\frac{n}{2}\rfloor}}}{\big(\mu + \frac{1}{2}\big)_{{\lfloor\frac{n}{2}\rfloor}}}
 P_{{\lfloor\frac{n}{2}\rfloor}}^{(\frac{d}{2},\mu-\frac{1}{2})}\big(1 - 2 \|\mathbf{x}\|^2\big).
\end{gather*}
In particular, setting $\mathbf{x} = \mathbf{0}$ we obtain
\begin{gather*}
K_n(u_\mu;\mathbf{0},\mathbf{0}) =
\frac{\big(\mu + \frac{d+1}{2}\big)_{{\lfloor\frac{n}{2}\rfloor}}}{\big(\mu + \frac{1}{2}\big)_{{\lfloor\frac{n}{2}\rfloor}}}
 P_{{\lfloor\frac{n}{2}\rfloor}}^{(\frac{d}{2},\mu-\frac{1}{2})}(1) =
\frac{\big(\mu + \frac{d+1}{2}\big)_{{\lfloor\frac{n}{2}\rfloor}}}{\big(\mu + \frac{1}{2}\big)_{{\lfloor\frac{n}{2}\rfloor}}}
\binom{{\lfloor\frac{n}{2}\rfloor} + \frac{d}{2}}{{\lfloor\frac{n}{2}\rfloor}}.\tag*{\qed}
\end{gather*}
\renewcommand{\qed}{}
\end{proof}

Summarizing the previous results, in next theorem we get the connection formulas for the orthogonal polynomials and kernels corresponding to the moment functionals $u_{\mu}$ and $v$.

\begin{Theorem}
Let us assume $v$ is quasi-definite and define the polynomials
\begin{gather*}
Q_{j,k}^{n}(\mathbf{x}) = \begin{cases}
P_{{\lfloor\frac{n}{2}\rfloor},1}^{n}(\mathbf{x})- a_{n} K_{n-1}(u_\mu;\mathbf{x},\mathbf{0})
&\text{if $n$ is even, $j = {\lfloor\frac{n}{2}\rfloor}$ and $k=1$,}\\
P_{j,k}^{n}(\mathbf{x}) &\text{in any other case,}
\end{cases}
\end{gather*}
where
\begin{gather*}
a_n = \frac{\lambda P_{{\lfloor\frac{n}{2}\rfloor},1}^{n}(\mathbf{0})}{1 + \lambda K_{n-1}(u_\mu;\mathbf{0},\mathbf{0})}.
\end{gather*}
Then according to Theorem~{\rm \ref{main-theorem}}, $\{Q_{j,k}^{n}(\mathbf{x})\}$ constitutes an orthogonal basis with respect to the moment functional~$v$. Moreover, the corresponding kernels are related by the expression
\begin{gather} \label{kernels-u-v}
K_n(v;\mathbf{x},\mathbf{y}) = K_n(u_\mu;\mathbf{x},\mathbf{y}) -
b_n P_{{\lfloor\frac{n}{2}\rfloor}}^{(\frac{d}{2},\mu-\frac{1}{2})}\big(1 - 2 \|\mathbf{x}\|^2\big)
 P_{{\lfloor\frac{n}{2}\rfloor}}^{(\frac{d}{2},\mu-\frac{1}{2})}\big(1 - 2 \|\mathbf{y}\|^2\big),
\end{gather}
where
\begin{gather*}
b_n = \frac{\lambda}{1 + \lambda K_n(u_\mu;\mathbf{0},\mathbf{0})} \left[\frac{\big(\mu + \frac{d+1}{2}\big)_{{\lfloor\frac{n}{2}\rfloor}}}{\big(\mu +
\frac{1}{2}\big)_{{\lfloor\frac{n}{2}\rfloor}}}\right]^2.
\end{gather*}
\end{Theorem}

As a consequence of \eqref{kernels-u-v} we can easily deduce the asymptotic behaviour of the Christof\/fel functions. From now on we will assume $\lambda>0$, also we have to impose the restriction \mbox{$\mu \geqslant 0$} because existing asymptotics for Christof\/fel functions in the classical case have only been established for this range of~$\mu$. The symbol $C$ will represent a constant but it does not have always the same value.

First, we get the asymptotics for the interior of the ball. For this purpose, we will need the following classical estimate for Jacobi polynomials
(see \cite[equations~(4.1.3) and (7.32.5)]{Sz78}):

\begin{Lemma} \label{szego-est}
For arbitrary real numbers $\alpha > -1$, $\beta >-1$ and $t \in [0, 1]$,
\begin{gather*}
\big| P_n^{(\alpha, \beta)}(t) \big| \leqslant C n^{-\frac{1}{2}}\big(1 - t + n^{-2}\big)^{-(\alpha+1/2)/2}.
\end{gather*}
And an analogous estimate on $[-1, 0]$ follows from $P_n^{(\alpha, \beta)}(t) = (-1)^n P_n^{(\beta, \alpha)}(-t)$.
\end{Lemma}

\begin{Theorem}
Let $r=\|\mathbf{x}\|$. For $0 < r \leqslant 1/2$ we have
\begin{gather} \label{left_as}
0 < K_n(u_\mu;\mathbf{x},\mathbf{x}) - K_n(v;\mathbf{x},\mathbf{x}) \leqslant C n^{-1} \left( 2r^{2}+\frac{4}{n^{2}}\right)^{-\frac{d+1}{2}}.
\end{gather}
For $1/2 \leqslant r < 1$ we have
\begin{gather} \label{right-as}
0 < K_n(u_\mu;\mathbf{x},\mathbf{x}) - K_n(v;\mathbf{x},\mathbf{x}) \leqslant C n^{-1} \left( 2\big(1-r^{2}\big)+\frac{4}{n^{2}}\right)^{-\mu}.
\end{gather}
Here $C$ is independent of $n$ and $x$. Consequently if $\mu \geqslant 0$, uniformly for $\mathbf{x}$ in compact subsets of $\{\mathbf{x}\colon 0<\|\mathbf{x}\| <1\} $,
\begin{gather}\label{asymp1}
\lim_{n\rightarrow \infty }K_n(v;\mathbf{x},\mathbf{x}) / \binom{n+d}{d} =\frac{1}{\sqrt{\pi }} \frac{\Gamma \big( \mu +\frac{1}{2}\big) \Gamma \big( \frac{d+1}{2}\big) }{\Gamma \big( \mu +\frac{d+1}{2}\big) } \big(1-\|\mathbf{x}\|^{2}\big)^{-\mu }.
\end{gather}
\end{Theorem}

\begin{proof}
From \eqref{kernels-u-v} we get
\begin{gather*}
0 < K_n(u_\mu;\mathbf{x},\mathbf{x}) - K_n(v;\mathbf{x},\mathbf{x}) =
b_n \left(P_{{\lfloor\frac{n}{2}\rfloor}}^{(\frac{d}{2},\mu-\frac{1}{2})}\big(1 - 2 \|\mathbf{x}\|^2\big)\right)^2
\end{gather*}
with
\begin{gather*}
0 < b_n = \frac{\lambda}{1 + \lambda \frac{\big(\mu + \frac{d+1}{2}\big)_{{\lfloor\frac{n}{2}\rfloor}}}{\big(\mu +
\frac{1}{2}\big)_{{\lfloor\frac{n}{2}\rfloor}}}
\left(\begin{matrix}\lfloor\frac{n}{2}\rfloor + \frac{d}{2}\\ \lfloor\frac{n}{2}\rfloor\end{matrix}\right)}
\left[\frac{\big(\mu + \frac{d+1}{2}\big)_{{\lfloor\frac{n}{2}\rfloor}}}{\big(\mu + \frac{1}{2}\big)_{{\lfloor
\frac{n}{2}\rfloor}}}\right]^2.
\end{gather*}
Using Stirling's formula we can easily deduce the convergence of the sequence $\{b_n\}$ to a positive value. Therefore we can f\/ind a constant $C$ such that $0 < b_n < C$, $n = 0, 1, 2, \dots$, that is
\begin{gather*}
0 < K_n(u_\mu;\mathbf{x},\mathbf{x}) - K_n(v;\mathbf{x},\mathbf{x}) < C
\left(P_{{\lfloor\frac{n}{2}\rfloor}}^{(\frac{d}{2},\mu-\frac{1}{2})}\big(1 - 2 \|\mathbf{x}\|^2\big)\right)^2,
\end{gather*}
and Lemma~\ref{szego-est} gives \eqref{left_as} and \eqref{right-as}. Finally, \eqref{asymp1} follows from the corresponding asymptotic result for Christof\/fel functions on the ball obtained by Y.~Xu in~\cite{X96}.
\end{proof}

At the origin the asymptotic is completely dif\/ferent, in fact, we recover the value of the mass~$\lambda$ as the limit of the Christof\/fel functions evaluated at $\mathbf{x} = \mathbf{0}$, as we show in our next result.

\begin{Theorem}
\begin{gather*}
\lim_{n\rightarrow \infty }K_n(v;\mathbf{0},\mathbf{0}) = \frac{1}{\lambda}.
\end{gather*}
\end{Theorem}

\begin{proof}
From \eqref{kernel} we get
\begin{align*}
K_n(v;\mathbf{0},\mathbf{0}) &=
K_n(u_{\mu};\mathbf{0},\mathbf{0}) - \frac{\lambda \left[K_n(u_{\mu};\mathbf{0},\mathbf{0})\right]^2}
{1 + \lambda K_n(u_{\mu};\mathbf{0},\mathbf{0})}
= \frac{K_n(u_{\mu};\mathbf{0},\mathbf{0})}
{1 + \lambda K_n(u_{\mu};\mathbf{0},\mathbf{0})},
\end{align*}
and the result follows from \eqref{Kn0} and Stirling's formula.
\end{proof}

\subsection[Uvarov modif\/ication of bivariate Bessel--Laguerre orthogonal polynomials]{Uvarov modif\/ication of bivariate Bessel--Laguerre orthogonal\\ polynomials}

As usual, let $\{L_n^{(\alpha)}(t)\}_{n\geqslant 0}$ denote the classical Laguerre polynomials orthogonal with respect to the moment functional
\begin{gather*}
\big\langle \ell^{(\alpha)}, f \big\rangle = \int_{0}^{+\infty} f(t) t^\alpha e^ {-t} dt, \qquad \alpha >-1,
\end{gather*}
normalized by the condition (equation~(22.2.12) in \cite{AS72})
\begin{gather*}
h_n^{(\alpha)} = \big\langle \ell^{(\alpha)},(L_n^{(\alpha)}(t))^2\big\rangle = \frac{\Gamma(\alpha+n+1)}{n!}.
\end{gather*}
Following \cite{KF49}, let $\{B_n^{(a,b)}(z)\}_{n\geqslant0}$ denote the univariate classical Bessel polynomials orthogonal with respect to the non positive-def\/inite moment functional
\begin{gather*}
\langle b^{(a,b)}, f \rangle = \int_{c} f(z) (2\pi i)^{-1} z^{a-2} e^{-b/z} dz, \qquad a\neq 0, -1,-2, \ldots, \qquad b\ne 0,
\end{gather*}
where $c$ is the unit circle oriented in the counter-clockwise direction, normalized by the condition
\begin{gather*}
B_n^{(a,b)}(0)=1,
\end{gather*}
and \cite[equation~(58), p.~113]{KF49}
\begin{gather*}
h_{n}^{(a,b)}=\big\langle b^{(a,b)}, (B_{n}^{(a,b)}(z))^2\big\rangle =\frac{(-1)^{n+1} n! b}{(2n+a-1) (a)_{n-1}}.
\end{gather*}

In \cite{KLL01}, the Krall and Shef\/fer's partial dif\/ferential equation (5.55)
\begin{gather*}
x^2 w_{xx} + 2 x y w_{xy} + \big(y^2-y\big) w_{yy} +g(x-1) w_x + g(y-\gamma)w_y = n (n+g-1)w,
\end{gather*}
was considered and the authors proved that it has an OPS $\big\{P^{(g, \gamma)}_{n,m}(x,y)\colon 0\leqslant m\leqslant n\big\}_{n\geqslant0}$ for $g \gamma+n\neq 0$, for $n\geqslant 0$. The corresponding moment functional is not positive-def\/inite.

Moreover, they proved that an explicit expression for the polynomial solutions is given by
\begin{gather}\label{BL}
P^{(g, \gamma)}_{n,m}(x,y) = B_{n-m}^{(g+2m, -g)}(x) x^m L_{m}^{(g\gamma-1)}\left(\frac{g y}{x}\right),
\end{gather}
for $g\neq 0$, $g\gamma >-2$, $g+n\neq 0$ and $g \gamma+n\neq 0$, when $n\geqslant 0$. These polynomials are orthogonal with respect to
the moment functional $u^{(g,\gamma)}$ def\/ined as (see Theorem 3.4 in \cite{KLL01})
\begin{gather*}
\big\langle u^{(g,\gamma)}, x^h y^k\big\rangle = \big\langle x^{k} \ell_x^{(g,-g)},x^h\big\rangle \big\langle b_y^{(g\gamma-1)}, y^k\big\rangle, \qquad
h,k\geqslant 0.
\end{gather*}

For the bivariate polynomials \eqref{BL}, we get
\begin{gather*}
h^{(g,\gamma)}_{n,m} = \big\langle u^{(g,\gamma)},\big(P^{(g,\gamma)}_{n,m}(x,y)\big)^2\big\rangle\\
\hphantom{h^{(g,\gamma)}_{n,m}}{} =
\big\langle b_x^{(g+2m,-g)}, \big(B_{n-m}^{(g+2m, -g)}(x)\big)^2\big\rangle
\big\langle \ell^{(g\gamma-1)}, \big(L_{m}^{(g\gamma-1)}(y)\big)^2\big\rangle\\
\hphantom{h^{(g,\gamma)}_{n,m}}{}= h_{n-m}^{(g+2m, -g)} h_m^{(g\gamma-1)}.
\end{gather*}

Now, we will modify the above moment functional adding a mass at the point $\mathbf{0}=(0,0)$. Thus, we def\/ine
\begin{gather*}
\big\langle v^{(g,\gamma)}, p(x,y)\big\rangle = \big\langle u^{(g,\gamma)},p(x,y)\big\rangle + \lambda p(0,0), \qquad \forall\, p\in\Pi^2.
\end{gather*}
Following Section \ref{section3}, since the polynomial $x^m L_{m}^{(g\gamma-1)}\left(\frac{g y}{x}\right)$ vanishes at $(0,0)$ except for $m=0$, we get
\begin{gather*}
P^{(g, \gamma)}_{n,0}(0,0) = B_{n}^{(g, -g)}(0) = 1,\qquad
P^{(g, \gamma)}_{n,m}(0,0) = 0,\qquad 0<m\leqslant n.
\end{gather*}
Therefore, \eqref{sP} becomes
\begin{gather*}
\mathsf{P}_n((0,0))= (1, 0, \dots, 0)^t,
\end{gather*}
and \eqref{sK} can be explicitly computed as follows
\begin{gather*}
\mathsf{K}_{n}((0,0),(x,y)) = K_{n}(u^{(g,\gamma)};(0,0),(x,y))\\
\hphantom{\mathsf{K}_{n}((0,0),(x,y))}{} = \sum_{m=0}^n\sum_{k=0}^m
\frac{P^{(g, \gamma)}_{m,k}(0,0) P^{(g, \gamma)}_{m,k}(x,y)}{h^{(g, \gamma)}_{m,k}} =
\sum_{m=0}^n \frac{B_{m}^{(g, -g)}(0)B_{m}^{(g, -g)}(x)}{h^{(g,-g)}_{m} h_0^{(g\gamma-1)}}\\
\hphantom{\mathsf{K}_{n}((0,0),(x,y))}{} = \frac{1}{h_0^{(g\gamma-1)}} K_n(b^{(g,-g)}; 0,x)
= \frac{1}{g \Gamma(g\gamma)}\sum_{m=0}^n \frac{(-1)^m(2m+g-1)(g)_{m-1}}{m!}\\
\hphantom{\mathsf{K}_{n}((0,0),(x,y))}{} = \frac{1}{g \Gamma(g\gamma)}\frac{(-1)^n (g)_n}{n!} =
\frac{(-1)^n}{g \Gamma(g\gamma)}\binom{g+n-1}{n}.
\end{gather*}
Then, using Theorem \ref{main-theorem}, $v^{(g,\gamma)}$ is quasi-def\/inite if and only if we choose $\lambda\in\mathbb{R}$ such that the value
\begin{gather*}
\lambda_n = 1 + \lambda \mathsf{K}_{n}((0,0),(x,y))= 1 + \lambda \frac{(-1)^n}{g \Gamma(g\gamma)}\binom{g+n-1}{n}
\end{gather*}
is dif\/ferent from zero for $n\geqslant 0$. Using the above computations, there exists only a numerable set of values of $\lambda\in \mathbb{R}$ such that the modif\/ied moment functional $v^{(g,\gamma)}$ is not quasi-def\/inite.

Therefore, let $\lambda \in\mathbb{R}$ such that $\lambda_n\neq 0$, $n\geqslant 0$ and consider the quasi-def\/inite moment functional $v^{(g,\gamma)}$. The modif\/ied polynomials are given by
\begin{gather*}
Q^{(g, \gamma)}_{n,m}(x,y) = P^{(g, \gamma)}_{n,m}(x,y), \qquad 0 < m \leqslant n,\\
Q^{(g, \gamma)}_{n,0}(x,y) = P^{(g, \gamma)}_{n,0}(x,y) - \frac{\lambda}{\lambda_{n-1}}\mathsf{K}_{n-1}((0,0),(x,y)).
\end{gather*}
We can compute explicitly $Q^{(g, \gamma)}_{n,0}(x,y)$. If $\tilde{\lambda} = \lambda/h_0^{(g\gamma-1)}$, then
\begin{gather*}
Q^{(g, \gamma)}_{n,0}(x,y) = B_n^{(g, -g)}(x) - \frac{\tilde{\lambda}}{1+\tilde{\lambda} K_n(b^{(g,-g)}; 0,0)} K_{n-1}\big(b^{(g,-g)}; 0,x\big),
\end{gather*}
that is, $Q^{(g, \gamma)}_{n,0}(x,y)$ coincides with the (univariate) orthogonal polynomial associated with the univariate modif\/ication of the Bessel moment functional given by
\begin{gather*}
\big\langle \tilde{b}^{(g,-g)}, p(x)\big\rangle = \big\langle b^{(g,-g)}, p(x)\big\rangle + \tilde{\lambda} p(0).
\end{gather*}
As a conclusion, the perturbation $v^{(g,\gamma)}$ of the bivariate moment functional $u^{(g,\gamma)}$ only af\/fects the bivariate polynomials $Q^{(g, \gamma)}_{n,0}(x,y)$, $n\geqslant 0$.

\subsection*{Acknowledgements}

The authors are really grateful to the anonymous referees for their valuable suggestions and comments which led us to improve this paper. This work has been partially supported by MINECO of Spain and the European Regional Development Fund (ERDF) through grant MTM2014--53171--P, and Jun\-ta de Andaluc\'{\i}a grant P11--FQM--7276 and research group FQM--384.

\pdfbookmark[1]{References}{ref}
\LastPageEnding


\begin{thebibliography}{99}
\footnotesize\itemsep=0pt

\bibitem{AS72}
Abramowitz M., Stegun I.A., Handbook of mathematical functions, with formulas,
 graphs, and mathematical tables, Dover Publications, New York, 1972.

\bibitem{APPR14}
Alfaro M., Pe{\~n}a A., P{\'e}rez T.E., Rezola M.L., On linearly related
 orthogonal polynomials in several variables, \href{http://dx.doi.org/10.1007/s11075-013-9747-2}{\textit{Numer. Algorithms}}
 \textbf{66} (2014), 525--553, \href{http://arxiv.org/abs/1307.5999}{arXiv:1307.5999}.

\bibitem{AM16}
Ariznabarreta G., Ma\~nas M., Multivariate orthogonal polynomial and integrable
 systems, \href{http://dx.doi.org/10.1016/j.aim.2016.06.029}{\textit{Adv. Math.}} \textbf{302} (2016), 628--739,
 \href{http://arxiv.org/abs/1409.0570}{arXiv:1409.0570}.

\bibitem{AM2}
Ariznabarreta G., Ma\~nas M., Darboux transformations for multivariate
 orthogonal polynomials, \href{http://arxiv.org/abs/1503.04786}{arXiv:1503.04786}.

\bibitem{AM3}
Ariznabarreta G., Ma\~nas M., Linear spectral transformations for multivariate
 orthogonal polynomials and multispectral Toda hierarchies,
 \href{http://arxiv.org/abs/1511.09129}{arXiv:1511.09129}.

\bibitem{BSX95}
Berens H., Schmid H.J., Xu Y., Multivariate {G}aussian cubature formulae,
 \href{http://dx.doi.org/10.1007/BF01193547}{\textit{Arch. Math. (Basel)}} \textbf{64} (1995), 26--32.

\bibitem{BM04}
Bueno M.I., Marcell{\'a}n F., Darboux transformation and perturbation of linear
 functionals, \href{http://dx.doi.org/10.1016/j.laa.2004.02.004}{\textit{Linear Algebra Appl.}} \textbf{384} (2004), 215--242.

\bibitem{Ch78}
Chihara T.S., An introduction to orthogonal polynomials, \textit{Mathematics
 and its Applications}, Vol.~13, Gordon and Breach Science Publishers, New
 York~-- London~-- Paris, 1978.

\bibitem{Ch1858}
Christof\/fel E.B., \"{U}ber die {G}au\ss ische {Q}uadratur und eine
 {V}erallgemeinerung derselben, \href{http://dx.doi.org/10.1515/crll.1858.55.61}{\textit{J.~Reine Angew. Math.}} \textbf{55}
 (1858), 61--82.

\bibitem{DFPP12}
Delgado A.M., Fern{\'a}ndez L., P{\'e}rez T.E., Pi{\~n}ar M.A., On the {U}varov
 modif\/ication of two variable orthogonal polynomials on the disk,
 \href{http://dx.doi.org/10.1007/s11785-011-0192-8}{\textit{Complex Anal. Oper. Theory}} \textbf{6} (2012), 665--676.

\bibitem{DFPPX10}
Delgado A.M., Fern{\'a}ndez L., P{\'e}rez T.E., Pi{\~n}ar M.A., Xu Y.,
 Orthogonal polynomials in several variables for measures with mass points,
 \href{http://dx.doi.org/10.1007/s11075-010-9391-z}{\textit{Numer. Algorithms}} \textbf{55} (2010), 245--264, \href{http://arxiv.org/abs/0911.2818}{arXiv:0911.2818}.

\bibitem{DX14}
Dunkl C.F., Xu Y., Orthogonal polynomials of several variables,
 \href{http://dx.doi.org/10.1017/CBO9781107786134}{\textit{Encyclopedia of Mathematics and its Applications}}, Vol.~155, 2nd ed.,
 Cambridge University Press, Cambridge, 2014.

\bibitem{FPPX10}
Fern{\'a}ndez L., P{\'e}rez T.E., Pi{\~n}ar M.A., Xu Y., Krall-type orthogonal
 polynomials in several variables, \href{http://dx.doi.org/10.1016/j.cam.2009.02.067}{\textit{J.~Comput. Appl. Math.}}
 \textbf{233} (2010), 1519--1524, \href{http://arxiv.org/abs/0801.0065}{arXiv:0801.0065}.

\bibitem{GvL}
Golub G.H., Van~Loan C.F., Matrix computations, 3rd ed., \textit{Johns Hopkins Studies
 in the Mathematical Sciences}, Johns Hopkins University Press, Baltimore, MD,
 1996.

\bibitem{HJ85}
Horn R.A., Johnson C.R., Matrix analysis, \href{http://dx.doi.org/10.1017/CBO9780511810817}{Cambridge University Press},
 Cambridge, 1985.

\bibitem{Kr80}
Krall A.M., Orthogonal polynomials satisfying fourth order dif\/ferential
 equations, \href{http://dx.doi.org/10.1017/S0308210500015213}{\textit{Proc. Roy. Soc. Edinburgh Sect.~A}} \textbf{87} (1981),
 271--288.

\bibitem{KF49}
Krall H.L., Frink O., A new class of orthogonal polynomials: {T}he {B}essel
 polynomials, \href{http://dx.doi.org/10.1090/S0002-9947-1949-0028473-1}{\textit{Trans. Amer. Math. Soc.}} \textbf{65} (1949), 100--115.

\bibitem{KS67}
Krall H.L., Shef\/fer I.M., Orthogonal polynomials in two variables, \href{http://dx.doi.org/10.1007/BF02412238}{\textit{Ann.
 Mat. Pura Appl.~(4)}} \textbf{76} (1967), 325--376.

\bibitem{KLL01}
Kwon K.H., Lee J.K., Littlejohn L.L., Orthogonal polynomial eigenfunctions of
 second-order partial dif\/ferential equations, \href{http://dx.doi.org/10.1090/S0002-9947-01-02784-2}{\textit{Trans. Amer. Math. Soc.}}
 \textbf{353} (2001), 3629--3647.

\bibitem{Ma85}
Maroni P., Sur quelques espaces de distributions qui sont des formes
 lin\'eaires sur l'espace vectoriel des polyn\^omes, in Orthogonal Polynomials
 and Applications ({B}ar-le-{D}uc, 1984), \href{http://dx.doi.org/10.1007/BFb0076543}{\textit{Lecture Notes in Math.}},
 Vol.~1171, Springer, Berlin, 1985, 184--194.

\bibitem{Ma88}
Maroni P., Le calcul des formes lin\'eaires et les polyn\^omes orthogonaux
 semi-classiques, in Orthogonal Polynomials and their Applications ({S}egovia,
 1986), \href{http://dx.doi.org/10.1007/BFb0083367}{\textit{Lecture Notes in Math.}}, Vol.~1329, Springer, Berlin, 1988,
 279--290.

\bibitem{Ma91}
Maroni P., Une th\'eorie alg\'ebrique des polyn\^omes orthogonaux.
 {A}pplication aux polyn\^omes orthogonaux semi-classiques, in Orthogonal
 Polynomials and their Applications ({E}rice, 1990), \textit{IMACS Ann.
 Comput. Appl. Math.}, Vol.~9, Baltzer, Basel, 1991, 95--130.

\bibitem{MP16}
Mart{\'{\i}}nez C., Pi{\~n}ar M.A., Orthogonal polynomials on the unit ball and
 fourth-order partial dif\/ferential equations, \href{http://dx.doi.org/10.3842/SIGMA.2016.020}{\textit{SIGMA}} \textbf{12}
 (2016), 020, 11~pages, \href{http://arxiv.org/abs/1511.07056}{arXiv:1511.07056}.

\bibitem{M81}
Mysovskikh I.P., Interpolation cubature formulas, Nauka, Moscow, 1981.

\bibitem{SX94}
Schmid H.J., Xu Y., On bivariate {G}aussian cubature formulae, \href{http://dx.doi.org/10.2307/2160762}{\textit{Proc.
 Amer. Math. Soc.}} \textbf{122} (1994), 833--841.

\bibitem{Sz78}
Szeg\H{o} G., Orthogonal polynomials, \textit{American Mathematical Society,
 Colloquium Publications}, Vol.~23, 4th~ed., Amer. Math. Soc., Providence,
 R.I., 1975.

\bibitem{Uv69}
Uvarov V.B., The connection between systems of polynomials that are orthogonal
 with respect to dif\/ferent distribution functions, \href{http://dx.doi.org/10.1016/0041-5553(69)90124-4}{\textit{USSR Comput. Math.
 Math. Phys.}} \textbf{9} (1969), no.~6, 25--36.

\bibitem{X92}
Xu Y., Gaussian cubature and bivariate polynomial interpolation, \href{http://dx.doi.org/10.2307/2153073}{\textit{Math.
 Comp.}} \textbf{59} (1992), 547--555.

\bibitem{X94}
Xu Y., On zeros of multivariate quasi-orthogonal polynomials and {G}aussian
 cubature formulae, \href{http://dx.doi.org/10.1137/S0036141092237200}{\textit{SIAM~J. Math. Anal.}} \textbf{25} (1994),
 991--1001.

\bibitem{X96}
Xu Y., Asymptotics for orthogonal polynomials and {C}hristof\/fel functions on a
 ball, \href{http://dx.doi.org/10.4310/MAA.1996.v3.n2.a6}{\textit{Methods Appl. Anal.}} \textbf{3} (1996), 257--272.

\bibitem{Z97}
Zhedanov A., Rational spectral transformations and orthogonal polynomials,
 \href{http://dx.doi.org/10.1016/S0377-0427(97)00130-1}{\textit{J.~Comput. Appl. Math.}} \textbf{85} (1997), 67--86.

\bibitem{Z99}
Zhedanov A., A method of constructing {K}rall's polynomials, \href{http://dx.doi.org/10.1016/S0377-0427(99)00070-9}{\textit{J.~Comput.
 Appl. Math.}} \textbf{107} (1999), 1--20.

\end{thebibliography}
\end{document}